\documentclass[12pt]{article}

\usepackage{amsmath}
\usepackage{amssymb}
\usepackage{amsthm}
\usepackage{enumerate}
\usepackage{hyperref}
\usepackage[top=0.8 in, bottom=0.8 in, left=0.8 in, right=0.8 in]{geometry}
\usepackage{bm}

\newtheorem{lemma}{Lemma}[section]
\newtheorem{theorem}[lemma]{Theorem}
\newtheorem{proposition}[lemma]{Proposition}
\newtheorem{corollary}[lemma]{Corollary}

\theoremstyle{definition}

\newtheorem{example}[lemma]{Example}
\newtheorem{remark}[lemma]{Remark}
\newtheorem*{remark*}{Remark}

\newtheorem{mainthm}{Theorem}

\newtheorem{maincor}[mainthm]{Corollary}

\numberwithin{equation}{section}

\usepackage[style=alphabetic, maxalphanames=4, maxbibnames=5]{biblatex}
\newcounter{constantno}

\begin{document}

\title{On differential geometry of non-degenerate CR manifolds
}
\author{Yuxin Dong\footnote{Supported by NSFC Grant No. 12171091} \and Yibin Ren}
\date{}
\maketitle

\begin{abstract}
  In this paper, 
  we consider a non-degenerate CR manifold \((M,H(M),J)\) with a given pseudo-Hermitian \(1\)-form \(\theta\), 
  and endow the CR distribution \(H(M)\) with any Hermitian metric \(h\) instead of the Levi form \(L_{\theta }\). 
  This induces a natural Riemannian metric \(g_{h,\theta }\) on \(M\) compatible with the structure. 
  The synthetic object \((M,\theta ,J,h)\) will be called a pseudo-Hermitian manifold, 
  which generalizes the usual notion of pseudo-Hermitian manifold \((M,\theta ,J,L_{\theta })\) in the literature. 
  Our purpose is to investigate the differential-geometric aspect of pseudo-Hermitian manifolds. 
  By imitating Hermitian geometry, 
  we find a canonical connection on \((M,\theta ,J,h)\), 
  which generalizes the Tanaka-Webster connection on \((M,\theta ,J,L_{\theta })\). 
  We define the pseudo-K\"{a}hler \(2\)-form by \(g_{h,\theta }\) and \(J\); and introduce the notion of a pseudo-K\"{a}hler manifold, 
  which is an analogue of a K\"{a}hler manifold. 
  It turns out that \((M,\theta ,J,L_{\theta })\) is pseudo-K\"{a}hlerian. 
  Using the structure equations of the canonical connection, 
  we derive some curvature and torsion properties of a pseudo-Hermitian manifold, 
  in particular of a pseudo-K\"{a}hler manifold. 
  Then some known results in Riemannian geometry are generalized to the pseudo-Hermitian case. 
  These results include some Cartan type results. 
  As an application, 
  we give a new proof for the classification of Sasakian space forms.  
\end{abstract}

\section*{Introduction}

CR manifolds are a type of manifold that arise in the study of several complex variables and complex geometry.
They play an important role in these fields.
The geometry of CR manifolds goes back to Poincar\'{e} and received a great attention in the works of Cartan, Tanaka, Chern-Moser, and others.
In recent decades,
many studies have been conducted on geometric analysis problems in CR manifolds (cf. \cite{jacobowitz1990cr}, \cite{dragomir2007diff}, \cite{boyer2008sasaki}, \cite{case2021lichnerowicz}, \cite{afeltra2023variation}).

A non-degenerate CR manifold is a real \((2m+1)\)-dimensional orientable manifold \(M^{2m+1}\) endowed with a non-degenerate CR structure \((H(M),J_{b})\) of hypersurface type.
The pair \((H(M),J_{b})\) consists of a rank \(2m\) subbundle of \(TM\) with a formally integrable almost complex structure.
The subbundle \(H(M)\) can be described by a global nowhere vanishing section \( \theta \) of its conormal bundle in \(T^{\ast } (M)\),
that is,
\(H(M)=\ker \theta \).
This \(1\)-form \(\theta \) is usually called a pseudo-Hermitian structure on \(M\) in the literature.
However,
we argue that this widely used terminology is not very appropriate when compared with a Hermitian structure on a complex manifold.
We prefer to call it a pseudo-Hermitian \(1\)-form and refer the term pseudo-Hermitian structures to the geometric structures that will be presented below.
The non-degenerate condition on the CR structure means that the pseudo-Hermitian \(1\)-form \(\theta \) is a contact form,
which gives rise to the Reeb vector field \(\xi \) that is transversal to \(H(M)\),
and the Levi form \(L_{\theta }\) that is a non-degenerate \(J_{b}\)-invariant symmetric bilinear form on \(H(M)\).
Then \(M\) admits a natural pseudo-Riemannian metric \( g_{\theta }\) (the Webster metric),
extending \(L_{\theta }\) in such a way that \(TM\) is an orthogonal direct sum of \(H(M)\) and \(span\{\xi \}\).
Clearly this metric becomes a Riemannian metric when the CR manifold \((M,\theta ,J_{b})\) is strictly pseudoconvex.
An important discovery in CR geometry is the canonical connection \(\nabla \) on any non-degenerate CR manifold \(M\) (the Tanaka-Webster connection),
found independently by N. Tanaka \cite{tanaka1975} and S. Webster \cite{webster1978}.
This connection is a linear connection compatible with both the metric \(g_{\theta }\) and the CR structure.

In this paper,
we consider a non-degenerate CR manifold \((M,H(M),J_b)\) with a given pseudo-Hermitian \(1\)-form \(\theta \),
and endow the vector bundle \((H(M),J_{b})\) with any Hermitian metric \(h\) instead of the Levi form \(L_{\theta }\).
From \(\theta \) and \(h\), we can define a Riemannian metric on \(M\) by
\begin{equation*}
  g_{h,\theta }=\pi_{H}^{\ast }h+\theta \otimes \theta
\end{equation*}
where \(\pi_{H}:T(M)=H(M)\oplus \mathbb{R} \xi \rightarrow H(M)\) is the natural projection.
Set \(J=\pi _{H}^{\ast }(J_{b}) \).
Henceforth we will refer to the synthetic object \((M,\theta ,J,h)\) as a pseudo-Hermitian manifold with a pseudo-Hermitian structure \(\{\theta ,J,h\}\) and to \(g_{h,\theta }\) as a pseudo-Hermitian metric on \(M\).
Our purpose is to investigate the differential-geometric aspect of pseudo-Hermitian manifolds.
Notice that our terminology "pseudo-Hermitian manifold" is more general than that in the literature,
which usually indicates the synthetic object $(M,\theta ,J,L_{\theta })$.

One effective way to study pseudo-Hermitian manifolds is to imitate the methods used to study Hermitian manifolds.
First we want to find a canonical connection for any pseudo-Hermitian manifold \((M,\theta ,J,h)\).
Applying \(J\) to the complexified tangent bundle of \(M\),
we have the following decomposition

\begin{equation*}
  T(M) \otimes \mathbb{C}
  =H^{1,0}(M) \oplus H^{0,1}(M) \oplus \mathbb{C} \xi
\end{equation*}
which induces natural projections \(\pi_{+}: T(M)\otimes \mathbb{C} \rightarrow H^{1,0}(M)\),
\(\pi_{-}=\overline{\pi_{+}}\),
and \(\pi_{0}:T(M)\otimes \mathbb{C} \rightarrow \mathbb{C} \xi \).
A natural connection on \(M\ \)should preserve the above decomposition of \(TM\otimes \mathbb{C} \) and \(g_{h,\theta }\).
This is equivalent to requiring that it makes \(J\) and \(g_{h,\theta }\) parallel.
Any connection with these two properties is called pseudo-Hermitian.
We shall determine a canonical connection among all pseudo-Hermitian connections on \(M\) by considering the components of their torsions.
Let \(\nabla \) be any pseudo-Hermitian connection on \(M\) and \( T^{\nabla }(\cdot ,\cdot )\) be its torsion that is extended complex linearly to \(TM\otimes \mathbb{C}\).
Then \(\pi_{0}(T_{\nabla }(\cdot ,\cdot ))\) is completely clear by Lemmas \ref{lem:1} and \ref{lem:2}. 
The remaining torsion information of the connection is thus contained in the \((H(M)\otimes \mathbb{C})\)-valued torsion \(2\)-form \(\pi_{H}(T_{\nabla }(\cdot ,\cdot ))\). 
We write
\begin{equation*}
  \pi_{H}(T_{\nabla }(\cdot ,\cdot ))=\Theta +\overline{\Theta }
\end{equation*}
where \(\Theta =\pi_{+}(T^{\nabla }(\cdot ,\cdot ))\). 
Let \(\tau\) be the \(H(M)\)-valued \(1\)-form defined by
\begin{equation*}
\tau (\cdot )=T^{\nabla }(\xi ,\cdot ).
\end{equation*}
Then we have the following decomposition
\begin{equation*}
\Theta +\overline{\Theta }=2\theta \wedge \tau +\Theta^{(2,0)}+\Theta^{(1,1)}+\overline{\Theta^{(2,0)}+\Theta^{(1,1)}}
\end{equation*}
where
\begin{equation*}
  \Theta^{(2,0)}=\Theta (\pi_{+},\pi_{+}), \quad
  \Theta^{(1,1)}=\Theta (\pi_{+},\pi_{-})+\Theta (\pi_{-},\pi_{+})\text{\ }
\end{equation*}
Notice that \(\Theta^{(0,2)}:=\Theta (\pi_{-},\pi_{-})=0\) due to the integrability of \(J_{b}\).

The first main theorem in this paper is the following:

\begin{mainthm}
  \label{thm:1}                 
  Let \((M^{2m+1},\theta ,J,h)\) be a pseudo-Hermitian manifold.
Set \(\sigma =\tau \circ J+J\circ \tau \).
Given any \(H^{1,0}(M)\)-valued \((1,1)\)-form \(\Psi \) on \(M\),
there is a unique pseudo-Hermitian connection \(\nabla \) such that
  \begin{equation*}
    \Theta^{(1,1)}=\Psi
  \end{equation*}
  and
  \begin{equation*}
    h(\sigma (X),Y)=-h(X,\sigma (Y))
  \end{equation*}
  for any \(X,Y\in T(M)\).
\end{mainthm}

By Theorem \ref{thm:1}, we have

\begin{maincor}
  \label{cor:1}                 
  On any pseudo-Hermitian manifold,
  there is a unique pseudo-Hermitian connection with \(\Theta^{(1,1)}=0\) and
  \(h(\sigma (X),Y)=-h(X,\sigma (Y))\).
\end{maincor}

The connection in Corollary \ref{cor:1} will be referred to as the canonical pseudo-Hermitian connection of \(M\).
It is easy to see that if \(h=L_{\theta }\),
the canonical pseudo-Hermitian connection coincides with the Tanaka-Webster connection.

Analogous to the case of Hermitian manifolds,
we define the pseudo-K\"{a}hler form of \((M^{2m+1},\theta ,J,h)\) by \(\Omega _{h,\theta }=g_{h,\theta}(J\cdot ,\cdot )\).
Then \(M\) is said to be of pseudo-K\"{a}hler type or simply a pseudo-K\"{a}hler manifold if \(d\Omega _{h,\theta }=0\).
It is easy to see that if \(h=L_{\theta }\),
then \(\Omega _{L_{\theta },\theta }=d\theta \).
This implies that \((M,\mu \theta ,J,\lambda L_{\theta })\) for any positive constants \(\lambda \) and \(\mu \) is of pseudo-K\"{a}hler type.
Besides,
we find that any non-degenerate Hopf hypersurface with the induced pseudo-Hermitian structure in a K\"{a}hler manifold is automatically a pseudo-K\"{a}hler manifold,
and the pseudo-Hermitian metric \(g_{h,\theta }\) is not the Webster metric in general.

In this paper a pseudo-Hermitian manifold is always endowed with its canonical pseudo-Hermitian connection.
From this connection,
we have the corresponding curvature tensor,
and various related notions of curvature,
including horizontal sectional curvature,
pseudo-Hermitian sectional curvature,
pseudo-Hermitian Ricci curvature and scalar curvature.
It is useful to derive the structure equations of the connection,
which yield some relations between the curvature and torsion components.
Using the structure equations,
we obtain the following result.

\begin{mainthm}
  \label{thm:2}                 
  Let \((M^{2m+1},\theta ,J,h)\) be a pseudo-Hermitian manifold.
  Then \(M\) is of pseudo-K\"{a}hler type if and only if \(\Theta^{(2,0)}=0\),
  \(\sigma =0\) and \(\tau \) is symmetric with respect to \(g_{h,\theta }\).
\end{mainthm}

\begin{remark*}
  In \cite{dragomir2007diff},
  the authors introduced the so-called pure condition to describe the Tanaka-Webster connection on a non-degenerate CR manifold \((M,\theta ,J,L_{\theta })\).
  We find that a pseudo-Hermitian connection on a general pseudo-Hermitian manifold \((M^{2m+1},\theta ,J,h)\) satisfies the pure condition if and only if \(\Theta^{(2,0)}=\Theta^{(1,1)}=\Theta^{(0,2)}=0\) and \(\sigma =0\).
  Hence the above theorem shows that the canonical connection of a pseudo-K\"{a}hler manifold satisfies the pure condition too.
\end{remark*}

The geometric meaning of \(\tau =0\) is further explained by the following theorem,
which generalizes a known result for a non-degenerate CR manifold
\((M,\theta ,J,L_{\theta })\) (cf. \cite{webster1978}, \cite{dragomir2007diff}).

\begin{mainthm}
  \label{thm:3}                 
  Let \((M^{2m+1},\theta ,J,h)\) be a pseudo-Hermitian manifold of pseudo-K\"{a}hler type.
  Then the following conditions are equivalent:

  \begin{enumerate}[(i)]
  \item \(\tau =0\).

  \item \(\xi \) is an infinitesimal CR automorphism.

  \item \((M,\mathcal{F}_{\xi },h)\) is a Riemannian foliation,
    where \(\mathcal{F}_{\xi }\) denotes the Reeb foliation determined by \(\xi \).
  \end{enumerate}
\end{mainthm}

A strictly pseudoconvex CR manifold \((M,\theta ,J,L_{\theta })\) with positive definite Levi form \(L_{\theta }\) and \(\tau =0\) is called Sasakian.
Thus a pseudo-K\"{a}hler manifold with \(\tau =0\) may be regarded as a generalization of Sasakian manifolds.

It is known that sectional curvatures of Riemannian manifolds or
holomorphic sectional curvatures of K\"{a}hler manifolds
determine their curvature tensors completely (cf. \cite{kobayashi1996diff1,kobayashi1996diff}).
These results can be generalized as follows.

\begin{mainthm}
  \label{thm:4}                 
  Let \((M^{2m+1},\theta ,J,h)\) be a pseudo-K\"{a}hler manifold with \(\tau =0\).
  Then the pseudo-Hermitian sectional curvatures of \( M \) determine the curvature tensor completely.
\end{mainthm}

\begin{maincor}
  \label{cor:2}                 
  The pseudo-Hermitian sectional curvatures of a Sasakian manifold \((M^{2m+1},\theta ,J,L_{\theta })\) determine the curvature tensor completely.
\end{maincor}

Theorem \ref{thm:4} also yields that the curvature tensor of a pseudo-K\"{a}hler manifold \(M\)
with \(\tau =0\) and constant pseudo-Hermitian sectional curvature has a unique expression.

A theorem of Cartan (\cite{cartan1951lecon}) tells us that the metric of a Riemannian manifold is,
in some sense,
determined locally by its curvature.
To generalize this result to the pseudo-Hermitian case,
we first introduce the notion of pseudo-Hermitian isometry.
A diffeomorphism \(f: (M,\theta ,J,h)\rightarrow (N,\widehat{\theta },\widehat{J},\widehat{h})\) between two pseudo-Hermitian manifolds is called a \emph{pseudo-Hermitian isometry} if \(f^{\ast }g_{\widehat{h},\widehat{\theta }}=g_{h,\theta }\) and \(f^{\ast }\widehat{J}=J\).
Next,
we present the theory of geodesics of the canonical connection on a pseudo-Hermitian manifold,
which includes the exponential map and Jacobi fields with respect to the canonical connection.
Then we are able to establish a Cartan-type result for pseudo-Hermitian manifolds that can be briefly described as follows (see \S \ref{sec:cartan-type-theorems} for the detailed statements).

\begin{mainthm}
  \label{thm:13}                
  The pseudo-Hermitian structure \(\{\theta ,J,h\}\) of a pseudo-Hermitian manifold is determined locally by the curvature and torsion of the canonical connection.
\end{mainthm}

We would like to point out that for pseudo-Hermitian manifolds
whose exponential maps are global diffeomorphisms,
their pseudo-Hermitian structures are determined globally by their curvatures and torsions.
In \cite{tanno1969sasak},
Tanno showed that Sasakian space forms are isomorphic to three model spaces.
However,
his proof used some analytic hypothesis for both \(M\) and \(g_{\theta }\).
As an application of Theorem \ref{thm:13},
we will give a new proof for this classification result.

\section{Pseudo-Hermitian structures on CR manifolds} \label{sec:pseudo-herm-struct}

Let \(M^{2m+1}\) be a \((2m+1)\)-dimensional smooth orientable manifold.
A \emph{CR structure} on \(M^{2m+1}\) is a complex rank-\(m\) subbundle \(H^{1,0}(M)\) of the complexified tangent bundle \(T(M)\otimes \mathbb{C}\) satisfying \( H^{1,0}(M)\cap H^{0,1}(M)=\{0\}\) (\(H^{0,1}(M)=\overline{H^{1,0}(M)}\)) and the Frobenius integrability property
\begin{equation}
  \bigl[ \Gamma (H^{1,0}(M),\Gamma (H^{1,0}(M) \bigr]
  \subseteq \Gamma (H^{1,0}(M)).  \label{eqn0-1}
\end{equation}
The complex subbundle \(H^{1,0}(M)\) corresponds to a real rank \(2m\) subbundle
\(H(M) := \operatorname{Re} \{H^{1,0}(M)\oplus H^{0,1}(M)\}\)
of \(T(M)\),
which carries a complex structure \(J_{b}\) defined by
\begin{equation*}
  J_{b}(Z+\overline{Z})=i(Z-\overline{Z})
\end{equation*}
for any \(Z\in H^{1,0}(M)\).
Hence the CR structure can also be described equivalently by the pair \((H(M),J_{b})\).
The synthetic object
\( (M,H^{1,0}(M)) \) or \((M,H(M),J_{b})\) is called a \emph{CR manifold}.

The conormal bundle \(E\) of \(H(M)\) in \(T^{\ast }M\) is a real line bundle on \( M \),
whose fiber at each point \(x\in M\) is given by
\begin{equation*}
  E_{x}=\{\omega \in T_{x}^{\ast }M:\ker \omega \supseteq H_{x}(M)\}.
\end{equation*}
Since the CR distribution \(H(M)\) is orientable due to the complex structure \(J_{b}\),
it follows from the orientability of \(M\) that the real line bundle \(E\) is also orientable,
so \(E\) has globally defined nowhere vanishing sections.
Any such a section
\(\theta \in \Gamma (E\backslash \{0\})\)
is referred to as a \emph{pseudo-Hermitian \(1\)-form} on \(M\).

A pseudo-Hermitian \(1\)-form \(\theta \) on \(M\) gives a dual description of \( H(M)\),
that is,
\(H(M)=\ker \theta \).
The Levi form \(L_{\theta }\) corresponding to \(\theta \) is defined by
\begin{equation}
  L_{\theta }(X,Y)=d\theta (X,J_{b}Y)  \label{eqn1}
\end{equation}
for any \(X,Y\in H(M)\).
The integrability condition (\ref{eqn0-1}) of the CR structure
implies that \(L_{\theta }\) is \(J_{b}\)-invariant,
and thus symmetric.
We say that \((M,H^{1,0}(M))\) is nondegenerate
if the Levi form \( L_{\theta }\) is non-degenerate on \(H(M)\)
for some choice of pseudo-Hermitian \(1\)-form \(\theta \) on \(M\).
In particular,
if \(L_{\theta }\) is positive definite on \(H(M)\) for some \(\theta \),
then \((M,H^{1,0}(M))\) is said to be strictly pseudoconvex.

From now on let \((M^{2m+1},H^{1,0}(M))\) be a non-degenerate CR manifold
with a given pseudo-Hermitian \(1\)-form \(\theta \).
The non-degenerate condition for \( L_{\theta }\) means that \(\theta \) is a contact form.
Thus there is a unique vector field \(\xi \in \Gamma (T(M))\) such that
\begin{equation}
  \theta (\xi )=1, \quad i_{\xi} d\theta =0  \label{eqn1-1}
\end{equation}
where \(i_{\xi }\) denotes the interior product with respect to \(\xi \).
This vector field \(\xi \) is referred to as the \emph{Reeb vector field}
of the contact manifold \((M,\theta )\).
The collection of all its integral curves forms an oriented one-dimensional foliation
\(\mathcal{F}_{\xi }\) on \(M\),
which is called the \emph{Reeb foliation}.
Using Cartan's magic formula,
we obtain from (\ref{eqn1-1}) that
\begin{equation}
  \mathcal{L}_{\xi }\theta =0  \label{eqn1-2}
\end{equation}
where \(\mathcal{L}_{\xi }\) denotes the Lie derivative with respect to \(\xi \).
The first condition in \eqref{eqn1-1} also implies that \(\xi \) is transversal to \(H(M)\).
Thus \(T(M)\) admits a decomposition
\begin{equation}
  T(M)=H(M)\oplus V_{\xi }  \label{eqn2}
\end{equation}
where \(V_{\xi }:= \mathbb{R} \xi \) is a trivial line bundle on \(M\).
In terms of terminology from foliation theory, 
\(H(M)\) and \(V_{\xi }\) are also called the horizontal and vertical distributions respectively.
Whenever the decomposition \eqref{eqn2} is given,
the complex structure \(J_{b}\) on \(H(M)\) can be extended to
an endomorphism \(J\) of \(T(M)\) by requiring that
\begin{equation}
  J \mid_{H(M)}=J_{b}\text{ \ and \ }J\mid_{V_{\xi }}=0  \label{eqn2-1}
\end{equation}
where \(\mid \) denotes the fiberwise restriction.
Clearly \( \operatorname{Im}(J)=H(M)\) and \((J,\xi ,\theta )\)
is an almost contact structure in the following sense
\begin{align}
  \label{eqn2-2}
  \begin{aligned}
    J^{2}
    & = -I+\theta \otimes \xi, \quad \theta (\xi )=1, \\
    J \xi
    &= 0, \quad \theta \circ J=0.
  \end{aligned}
\end{align}
Actually the latter two conditions\ in \eqref{eqn2-2} are deducible from the former two
(cf. Theorem 4.1 in \cite{blair2002riem}).
Notice that in the definition of an almost contact structure,
\(J\) is not required to satisfy the integrability condition \eqref{eqn0-1} in general.
However we will restrict our attention only to the integrable case in this paper.
Let \(\pi_{H}:T(M)\rightarrow H(M) \) be the natural projection
associated with the direct sum decomposition \eqref{eqn2}.
Then the Levi form \(L_{\theta }\) may be extended to a pseudo-Riemannian metric as follows:
\begin{equation}
  g_{\theta }=\pi_{H}^{\ast }(L_{\theta }) + \theta \otimes \theta  \label{eqn3}
\end{equation}
where \(\pi_{H}^{\ast }(L_{\theta })\) denotes the \(2\)-tensor field on \(M\) given by
\begin{equation}
  (\pi_{H}^{\ast }L_{\theta })(X,Y)=L_{\theta }(\pi_{H}X,\pi_{H}Y)
  \label{eqn3-1}
\end{equation}
for any \(X,Y\in T(M)\).
The metric \(g_{\theta }\) in \eqref{eqn3} is called the \emph{Webster metric} of
\((M,\theta ,J)\).
In particular,
if \( L_{\theta }\) is positive definite on \(H(M)\),
then \(g_{\theta }\) is a Riemannian
metric on \(M\).

In this paper,
we will endow \((H(M),J_b)\) with a Hermitian metric instead of \(L_{\theta}\).
By definition,
a Hermitian metric on \((H(M),J_b)\) is a fiberwise metric \(h\) on \(H(M)\) satisfying
\begin{equation}
  h(J_bX,J_bY)=h(X,Y) \label{eqn3-2}
\end{equation}
for any \(X,Y\in H(M)\).
With such a metric,
\(\left( H^{1,0}(M),h\right) \) becomes a Hermitian vector bundle of
complex rank \(m\) over \(M^{2m+1}\).
In terms of \(h\) and \(\theta \),
one may introduce a Riemannian metric \( g_{h,\theta}\) on \((M,\theta)\) by
\begin{equation}
  g_{h,\theta}=\pi_H^{\ast}(h)+\theta \otimes \theta \label{eqn4}
\end{equation}
where \(\pi_H^{\ast}(h)\) is defined similarly as in \eqref{eqn3-1}.
This metric \(g_{h,\theta}\) is a Riemannian extension of \(h\).
Notice that the tensor field \(\pi_H^{\ast}(h)\) depends on the decomposition \eqref{eqn2} too.
A fiberwise metric on \(H(M)\) may have different Riemannian extensions if different decompositions of \(T(M)\) are considered.
In following,
we will often write \(\pi_H^{\ast}(h)\) as \(h\) for simplicity when the decomposition \eqref{eqn2} is fixed.
From \eqref{eqn4},
we have
\begin{equation}
  \theta (X)=g_{h,\theta}(\xi ,X) \label{eqn5}
\end{equation}
and
\begin{equation}
  g_{h,\theta}(JX,JY)=g_{h,\theta}(X,Y)-\theta (X)\theta (Y) \label{eqn5-1}
\end{equation}
for any \(X,Y\in T(M)\).
We find that the Riemannian metric \(g_{h,\theta}\) is compatible with the almost contact structure \((J,\xi ,\theta)\).
Thus \( (J,\xi ,\theta ,g_{h,\theta})\) is an almost contact metric structure,
which corresponds to a \(\{1\}\times U(m)\)-structure on \(M^{2m+1}\) (see page 44 in \cite{blair2002riem} for related definitions of these notions or Chapter 6 in \cite{boyer2008sasaki}).
The decomposition \eqref{eqn2} clearly becomes an orthogonal direct sum of \(TM\) with respect to \(g_{h,\theta}\).
When \(M\) is a strictly pseudoconvex CR manifold,
one natural Hermitian metric on \(H(M)\) is given by the positive definite Levi form \(L_{\theta}\).
In this case,
\(g_{L_{\theta},\theta}\) is just the Webster metric \(g_{\theta}\).
One usually uses the contact condition about \(\theta \) to induce the decomposition \eqref{eqn2} on an abstract CR manifold.
However,
in some situations,
this kind of decomposition is available without the non-degenerate condition of \(d\theta \) (see Example \ref{ex:2}).

Henceforth the synthetic object \((M^{2m+1},\theta ,J,h)\) will be called a pseudo-Hermitian manifold with pseudo-Hermitian structure \((\theta ,J,h)\).
Notice that the terminology "pseudo-Hermitian manifold" here is more general than that in the literature,
which usually indicates the synthetic object \((M^{2m+1},\theta ,J,L_{\theta})\).
Since \(L_{\theta}\) is not positive definite for a general non-degenerate CR manifold,
the corresponding Webster metric \(g_{\theta}\) is only a pseudo-Riemannian metric.
This causes some trouble for using tools in differential geometry.
Introducing a Hermitian metric \(h\) on \(H(M)\) allows us to investigate geometric analysis problems on non-degenerate (or more general) CR manifolds.
Even if \(M\) is strictly pseudoconvex,
we may explore other metric structures besides the Levi metrics.

We have already mentioned that a non-degenerate CR manifold \((M^{2m+1},\theta ,J)\) has a natural foliation \(\mathcal{F}_{\xi}\).
A smooth \(r\)-form \(\omega \) on \(M\) is said to be horizontal if \(i_{\xi}\omega =0\).
Let \(\mathcal{A}_H^r(M)\) denote the space of all horizontal smooth \(r\)-forms on \(M\).
Notice that \( \mathcal{A}_H^r(M)=\{0\}\) if \(r>2m\) (the rank of \(H(M)\)).
Using the exterior differentiation \(d\),
one may introduce a differential operator
\(d_H:\mathcal{A}_H^r(M)\rightarrow \mathcal{A}_H^{r+1}(M)\) (\(r=0,1,...,2m\)) by
\begin{equation}
  (d_{H}\omega )(X_{1},...,X_{r+1})
  = d\omega \left( \pi_{H}(X_{1}) , \cdots , \pi_{H}(X_{r+1})\right)  \label{eqn5-2}
\end{equation}
for any \(r\)-form \(\omega \in \mathcal{A}_{H}^{r}(M)\),
and \( X_{1}, \cdots ,X_{r+1}\in TM\).
Clearly \(d_{H}\) satisfies
\begin{equation}
  d_{H}(\omega \wedge \sigma)
  = d_{H}\omega \wedge \sigma +(-1)^{r}\omega
  \wedge d_{H}\sigma  \label{eqn5-2-1}
\end{equation}
for any \(\omega \in \mathcal{A}_{H}^{r}(M)\) and \(\sigma \in \mathcal{A}_{H}^{s}(M)\).
However,
the sequence \(\{\mathcal{A}_{H}^{r}(M),d_{H}\}_{r\geq 0}\) is not a cochain complex in general due to the non-integrability of \( H(M) \) as a (real) distribution.

Following the terminology in foliation theory (cf. \cite{molino1988folia}),
a smooth \(r\)-form \(\omega \) on \(M\) is said to be \emph{basic} if it satisfies
\begin{equation}
  i_{\xi}\omega =0 \quad \mbox{and} \quad i_{\xi } d\omega =0.  \label{eqn5.4}
\end{equation}
Let \(\mathcal{A}_{b}^{r}(M)\) denote the space of all basic smooth \(r\)-forms on \(M\).
It follows from the definition that if \(\omega \in \) \(\mathcal{A}_{b}^{r}(M)\),
then \(d\omega \in \mathcal{A}_{b}^{r+1}(M)\).
Clearly \(\{\mathcal{A}_{b}^{r}(M),d\}_{p\geq 0}\) is a cochain complex,
and thus one can define the basic cohomology
\begin{equation}
  \mathcal{H}_{b}^{r}(M)=\frac{\ker \{d:\mathcal{A}_{b}^{r}(M)\rightarrow
    \mathcal{A}_{b}^{r+1}(M)\}}{\operatorname{Im}\{d:\mathcal{A}_{b}^{r-1}(M)\rightarrow
    \mathcal{A}_{b}^{r}(M)\}}.  \label{eqn5.5}
\end{equation}
This notion was introduced by B. Reinhart (\cite{reinhart1959harmon}) for a general foliation.
Notice that \(d\theta \) is a basic form in view of the second property of \eqref{eqn1-1}.
Although \(d\theta \) corresponds to a trivial de Rham cohomology class,
its basic cohomology class \([d\theta ]_{b}\) may not be trivial
(see \S \ref{sec:pseudo-kahl-mfd} for related discussions).

On a pseudo-Hermitian manifold \((M^{2m+1},\theta ,J,h)\),
we define a \(2\)-form \(\Omega_{h,\theta }\) on \(M\) by
\begin{equation}
  \Omega_{h,\theta }(X,Y):=g_{h,\theta }(JX,Y)=h(J_{b}\pi_{H}X,\pi_{H}Y)
  \label{eqn6}
\end{equation}
for any \(X,Y\in T(M)\),
which will be called the associated \emph{pseudo-K\"{a}hler form} of \((\theta ,J,h)\).
It is easy to verify by \eqref{eqn3-2},
\eqref{eqn4} and \eqref{eqn6} that \(\Omega_{h,\theta }\) is \(J\)-invariant and
\begin{equation*}
  \theta \wedge \Omega_{h,\theta }^{m}=m!2^{m}dv_{g_{h,\theta }}
\end{equation*}
where \(dv_{g_{h,\theta }}\) denotes the volume form of \(g_{h,\theta }\).
In particular,
\eqref{eqn6} implies
\begin{equation}
  i_{\xi }\Omega_{h,\theta }=0,  \label{eqn6-1}
\end{equation}
that is,
\(\Omega_{h,\theta }\in \mathcal{A}_{H}^{2}(M)\).
If \(d_{H}\Omega_{h,\theta }=0\),
\((M^{2m+1},\theta ,J,h)\) is said to be of \emph{horizontally K\"{a}hler type}.
Furthermore,
if \(d\Omega_{h,\theta }=0\),
that is,
\begin{equation}
  i_{\xi }d\Omega_{h,\theta }=0 \quad \mbox{and} \quad d_{H}\Omega_{h,\theta }=0,  \label{eqn6-2}
\end{equation}
then the pseudo-Hermitian manifold is said to be of \emph{pseudo-K\"{a}hler type}.
For simplicity,
it is also called a pseudo-K\"{a}hler manifold.
In this case,
the closed basic \(2\)-from \(\Omega_{h,\theta }\) corresponds to a basic cohomology class \([\Omega_{h,\theta }]_{b}\in \mathcal{H}_{b}^{2}(M)\) ,
which will be called the pseudo-K\"{a}hler class of \(g_{h,\theta }\).
By \eqref{eqn6},
we see that \(\Omega_{\lambda h,\theta }=\lambda \Omega_{h,\theta }\) for any positive smooth function \(\lambda \) on \(M\).
Clearly \( (M,H^{1,0}(M),\theta ,\lambda h)\) is no longer pseudo-K\"{a}hlerian in general.
Suppose now that \(\lambda \) and \(\mu \) are any positive constants.
Then \(\frac{1}{\mu }\xi \) is the Reeb vector field of \(\mu \theta \).
This implies that \(\theta \) and \(\mu \theta \) induce the same decomposition \eqref {eqn2} of the tangent bundle.
The transformation from \(g_{h,\theta }\) to \( g_{\lambda h,\mu \theta }\) is called a \emph{mixed} \emph{homothetic transformation }of the metric.
In particular,
\(g_{\lambda h,\theta }\) is referred to as a \emph{horizontal homothetic transformation} of \(g_{h,\theta}\).
Under the mixed homothetic transformation,
we have \(\Omega_{\lambda h,\mu \theta }=\lambda \Omega_{h,\theta }\).
It follows that the mixed homothetic transformation preserves the pseudo-K\"{a}hlerian property.

\begin{example}
  \label{ex:1}                  
  Let \((M^{2m+1},\theta ,J)\) be a strictly pseudoconvex CR manifold with a positive definite \(L_{\theta }\).
  Then \eqref{eqn1} and \eqref{eqn6} yield
  \begin{equation}
    \Omega_{L_{\theta },\theta }=d\theta  \label{eqn7}
  \end{equation}
  that is,
  \((J,\xi ,\theta ,g_{L_{\theta },\theta })\) becomes the so-called \emph{contact metric structure}.
  Notice that the almost CR structure in the definition of a contact metric structure
  is not required to be integrable
  (see, e.g., \cite{blair2002riem} or Definition 6.4.4 in \cite{boyer2008sasaki}). 
  As pointed before,
  we restrict our attention to integrable CR structures in this paper.
  From the above discussion,
  we know that \((M,\mu \theta ,J,\lambda L_{\theta })\) with \(\lambda ,\mu\) positive constants is pseudo-K\"{a}hlerian.
  Here,
  if \(\lambda \neq \mu\),
  \(g_{\lambda L_{\theta },\mu \theta }\) is not the Webster metric.
\end{example}

\begin{example}
  \label{ex:2}                  
  Let \((N^{m+1},\widetilde{J},\widetilde{g})\) be a K\"{a}hler manifold of complex dimension \(m+1\).
  Its K\"{a}hler form \(\Phi \ \) is defined by \(\Phi (X,Y)=\widetilde{g}(\widetilde{J}X,Y)\) for any \(X,Y\in TN\).
  Let \(\mathcal{I} :M^{2m+1}\hookrightarrow N^{m+1}\) be a real hypersurface with induced metric \(g\) and unit normal vector field \(v\).
  Set \(\widehat{\xi }:=-\widetilde{J}v\) and \(\theta =g(\widehat{\xi },\cdot )\).
  Then the induced CR structure on \(M\) is given by \((H(M),J_{b})\) with \(H(M)=\ker \theta \) and \(J_{b}=\widetilde{J} \mid_{H(M)}\).
  We have a natural decomposition
  \begin{equation}
    T(M)=H(M)\oplus \mathbb{R}\widehat{\xi }.  \label{eqn7-1}
  \end{equation}
  without imposing any non-degenerate condition on \(L_{\theta }\).
  Notice that \( \widehat{\xi }\) does not satisfy \(i_{\widehat{\xi }}d\theta =0\) in general.
  This property plays some role in next section.

  We say that \(M\) is a Hopf hypersurface if \(\widehat{\xi }\) is a principal vector everywhere.
  Tubes over complex submanifolds in \(N\) are known to be Hopf (cf.
  \cite{niebergall1997real},
  page 244).
  It is easy to verify that \(i_{\widehat{\xi }}d\theta =0\) if and only if \(M\) is Hopf.
  Assume now that \(M\) is a non-degenerate Hopf hypersurface.
  Let \(h\) be the induced Hermitian metric on \(H(M)\) from \(g\).
  Clearly
  \begin{equation*}
    g_{h,\theta }=g=\mathcal{I}^{\ast }(\widetilde{g}) \quad \mbox{and} \quad
    \Omega_{h,\theta }=\mathcal{I}^{\ast }(\Phi ).
  \end{equation*}
  Since \(N\) is K\"{a}herian,
  that is,
  \(d\Phi =0\),
  we have \(d\Omega_{h,\theta}=0\) automatically.
  This shows that \(M\) is pseudo-K\"{a}hlerian.
\end{example}

\section{Canonical connections} \label{sec:canon-conn}

Let \((M,H^{1,0}(M))\) be a non-degenerate CR manifold and \(\theta \) a fixed pseudo-Hermitian \(1\)-form with Reeb vector field \(\xi \).
In this section,
we will investigate linear connections on \(M\).
If \(\nabla \) is a linear connection on \(M\),
its torsion \(T_{\nabla }\) is defined by
\begin{equation}
  T_{\nabla }(X,Y)=\nabla_{X}Y-\nabla_{Y}X-[X,Y]  \label{eqn8-0}
\end{equation}
for any \(X,Y\in \Gamma (T(M))\).
By complex linear extensions,
\(T_{\nabla }\) may be defined for any \(Z,W\in \Gamma (T(M)\otimes \mathbb{C})\).
Applying \(J\) to the complexified decomposition of \eqref{eqn2} gives
\begin{equation}
  T(M)\otimes \mathbb{C}=H^{1,0}(M)\oplus H^{0,1}(M)\oplus \mathbb{C}\xi .
  \label{eqn8}
\end{equation}
Consequently we have the following projections
\begin{align*}
  \pi_{H}
  & : T(M)\otimes \mathbb{C}\longrightarrow H(M)\otimes \mathbb{C=}
    H^{1,0}(M)\oplus H^{0,1}(M)\text{, \ \ } \\
  \pi_{+}
  & : T(M)\otimes \mathbb{C}\rightarrow H^{1,0}(M)\text{, \ }
    \pi_{-}:T(M)\otimes \mathbb{C}\rightarrow H^{0,1}(M)\text{,} \\
  \pi_{0}
  & : T(M)\otimes \mathbb{C}\rightarrow \mathbb{C\xi }.
\end{align*}
Clearly \(\pi_{H}=\pi_{+}+\pi_{-}\) and
\begin{align}
  \label{eqn8-1}
  \begin{aligned}
    \pi_{+}(Z)
    & =\frac{1}{2}\left( \pi_{H}(Z)-\sqrt{-1}J\pi_{H}(Z)\right) \\
    \pi_{-}(Z)
    & =\frac{1}{2}\left( \pi_{H}(Z)+\sqrt{-1}J\pi_{H}(Z)\right)
  \end{aligned}
\end{align}
for any \(Z\in T(M)\otimes \mathbb{C}\).
A natural connection on \((M,\theta ,J)\) should preserve the decomposition \eqref{eqn8}.
This is equivalent to requiring that the connection \(\nabla \) preserves both \( H(M) \) and \(J\),
that is,
\(\nabla_{X}\Gamma (H(M))\subseteq \Gamma (H(M))\) and \(\nabla_{X}J=0\) for any \(X\in \Gamma (TM)\).
Notice here that \(J\) is the extended endomorphism of $J_{b}$.

\begin{lemma}
  \label{lem:1}                 
  Let \(\nabla \) be a linear connection on \(M\)
  preserving the decomposition \eqref{eqn8} and let \(T_{\nabla }\) be its torsion field.
  Then for any \(Z,W\in H^{1,0}(M)\),
  we have
  \begin{enumerate}[(i)]
  \item  \(\pi_{0}((T_{\nabla }(Z,W))=\pi_{0}(T_{\nabla }(\overline{Z},\overline{W}))=0\) and
    \begin{equation*}
      \pi_{0}(T_{\nabla }(Z,\overline{W}))
      =2\sqrt{-1}L_{\theta }(Z,\overline{W})\xi .
    \end{equation*}

  \item \(\pi_{+}(T_{\nabla }(\overline{Z},\overline{W}))=\pi_{-}(T_{\nabla}(Z,W))=0\).
  \end{enumerate}
\end{lemma}

\begin{proof}
  Since \(\nabla \) preserves the decomposition \eqref{eqn8},
  we have
  \begin{equation*}
    \pi_{0}((T_{\nabla }(Z,W))=\pi_{0}(T_{\nabla }(\overline{Z},\overline{W}))=0
  \end{equation*}
  for any \(Z,W\in H^{1,0}(M)\).
  Using \(\ker \theta =H(M)\) and \eqref{eqn1},
  we deduce
  \begin{align*}
    \theta (T_{\nabla }(Z,\overline{W}))
    & = \theta (\nabla_{Z}\overline{W}
      -\nabla_{\overline{W}}Z-[Z,\overline{W}]) \\
    & = -\theta ([Z,\overline{W}])=2d\theta (Z,\overline{W}) \\
    & = 2\sqrt{-1}L_{\theta }(Z,\overline{W}).
  \end{align*}
  This completes the proof of (i).

  Next we prove (ii). The integrability of the CR structure implies that
  \begin{equation*}
    \pi_{+}(T_{\nabla }(\overline{Z},\overline{W}))
    =-\pi_{+}([\overline{Z}, \overline{W}])=0
  \end{equation*}
  for any \(Z,W\in H^{1,0}(M)\).
  Taking conjugation of \(\pi_{+}(T_{\nabla }(\overline{Z},\overline{W}))\),
  we get \(\pi_{-}(T_{\nabla }(Z,W))=0\).
\end{proof}

We now consider a pseudo-Hermitian manifold \((M,\theta ,J,h)\).
A connection on \(M\) preserving both the decomposition \eqref{eqn8}
and \( g_{h,\theta }\) will be called a \emph{pseudo-Hermitian connection}.

\begin{lemma}
  \label{lem:2}                 
  Let \(\nabla \) be a pseudo-Hermitian connection on \((M^{2m+1},\theta ,J,h)\).
  Then (i) \(\nabla \xi =0\) and \(\nabla \theta =0\).
  (ii) \(\pi_{0}(T_{\nabla }(\xi ,X))=0\ \)for any $X\in T(M)$.
\end{lemma}

\begin{proof}
  In terms of the properties of a pseudo-Hermitian connection,
  we see that \(\nabla \xi \) is parallel to \(\xi\), 
  and \(g_{h,\theta }(\nabla \xi ,\xi )=0\).
  This shows
  \begin{equation}
    \nabla \xi =0,  \label{eqn8-2}
  \end{equation}
  which implies that \((\nabla_{X}\theta )(\xi )=0\) for any \(X\in \Gamma (T(M)) \).
  It is also easy to verify that \((\nabla_{X}\theta )(Y)=0\) for any \(X\in \Gamma (T(M))\) and \(Y\in \Gamma (H(M))\).
  Hence
  \begin{equation}
    \nabla \theta =0.  \label{eqn8-3}
  \end{equation}

  Using \eqref{eqn8-2}, \eqref{eqn8-3} and \eqref{eqn1-2},
  we deduce that
  \begin{align*}
    \theta (T_{\nabla }(\xi ,X))
    & = \theta (\nabla_{\xi }X-\nabla_{X}\xi -[\xi
      ,X]) \\
    & = \theta (\nabla_{\xi }X)-\theta ([\xi ,X]) \\
    & = \nabla_{\xi }(\theta (X))-\theta ([\xi ,X]) \\
    & = (\mathcal{L}_{\xi }\theta )(X) \\
    & = 0
  \end{align*}
  for any \(X\in \Gamma (T(M))\).
  This proves (ii).
\end{proof}

Let \(\nabla \) be a pseudo-Hermitian connection in \((M,\theta ,J,h)\) .
According to Lemmas \ref{lem:1} and \ref{lem:2},
the \(\xi \)-component of \(T_{\nabla }(\cdot ,\cdot )\) with respect to the decomposition \eqref{eqn8} is completely clear.
The remaining torsion information of the connection is thus contained in the \(H(M)\)-valued \(2\)-form \(\pi_{H}(T_{\nabla }(\cdot ,\cdot ))\).
The complex linear extension of this \(2\)-form,
still denoted by the same notation,
may be written as
\begin{equation}
  \pi_{H}(T_{\nabla }(\cdot ,\cdot ))=\Theta +\overline{\Theta }
  \label{eqn8-4}
\end{equation}
where \(\Theta :=\pi_{+}(T_{\nabla }(\cdot ,\cdot ))\).
In order to investigate the pseudo-Hermitian connection \(\nabla \) further,
let us consider the following \(H^{1,0}(M)\)-valued \(2\)-forms
\begin{align}
  \Theta^{(2,0)}(Z,W)
  & = \Theta (\pi_{+}(Z),\pi_{+}(W))  \notag \\
  \Theta^{(0,2)}(Z,W)
  & = \Theta (\pi_{-}(Z),\pi_{-}(W))  \label{eqn8-5} \\
  \Theta^{(1,1)}(Z,W)
  & = \Theta (\pi_{+}(Z),\pi_{-}(W))+\Theta (\pi_{-}(Z),\pi_{+}(W))  \notag
\end{align}
for any \(Z,W\in T(M)\otimes \mathbb{C}\) and introduce a \(T(M)\)-valued \(1\) -form \(\tau\) by
\begin{equation}
  \tau (X)=T_{\nabla }(\xi ,X)  \label{eqn8-6}
\end{equation}
for any \(X\in T(M)\).
We have already shown in Lemma \ref{lem:1} that \(\Theta^{(0,2)}=0\),
which is just the integrability of the CR structure.
As a partial ingredient of \(T_{\nabla }\),
\(\tau \) will be referred to as the \emph{pseudo-Hermitian torsion} \(1\)-form.
Since \(T_{\nabla }\) is anti-symmetric,

\begin{equation}
  \tau (\xi )=0.  \label{eqn8-6-0}
\end{equation}
By Lemma \ref{lem:2},
\(\tau \) is actually an \(H(M)\)-valued \(1\)-form,
whose complex linear extension will be still denoted by \(\tau \).
It is evident from the above discussion that \(\pi_{H}(T_{\nabla }(\cdot ,\cdot ))\)
is completely determined by \(\Theta^{(2,0)}\),
\(\Theta^{(1,1)}\) and $\tau $.

Let \(\Psi \) be a complex valued \(2\)-form on \(M\).
We say that \(\Psi \) is a \( (1,1)\)-form with respect to the structure \(J\) if
\begin{equation*}
  \Psi (\pi_{+},\pi_{+})=\Psi (\pi_{-},\pi_{-})=0
  \quad \mbox{and} \quad
  i_{\xi}\Psi =0.
\end{equation*}
For example,
\(\Theta^{(1,1)}\) is a \((1,1)\)-form on \(M\).
Put
\begin{equation}
  \sigma =\tau \circ J+J\circ \tau .  \label{eqn8-6-1}
\end{equation}
Clearly \(\sigma \) is an \(H(M)\)-valued \(1\)-form with \(\sigma (\xi )=0\).
By \eqref{eqn2-2},
we have
\begin{equation}
  \sigma \circ J=J\circ \sigma =J\circ \tau \circ J-\tau  \label{eqn8-6-2}
\end{equation}
and thus
\begin{equation}
  \sigma (H^{1,0}(M))\subseteq H^{1,0}(M), \quad \sigma
  (H^{0,1}(M))\subseteq H^{0,1}(M).  \label{eqn8-6-3}
\end{equation}

\begin{theorem}
  \label{thm:5}                 
  Let \((M^{2m+1},\theta ,J,h)\) be a pseudo-Hermitian manifold.
  Given any \(H^{1,0}(M)\)-valued \((1,1)\)-form \(\Psi \) on \(M\),
  there is a unique pseudo-Hermitian connection \(\nabla \) such that
  \begin{equation}
    \Theta^{(1,1)}=\Psi  \label{eqn8-7}
  \end{equation}
  and
  \begin{equation}
    h(\sigma (X),Y)=-h(X,\sigma (Y))  \label{eqn8-8}
  \end{equation}
  for any \(X,Y\in T(M)\),
  where \(\sigma \) is defined by \eqref{eqn8-6-1}.
\end{theorem}

\begin{proof}
  First we show the uniqueness of the pseudo-Hermitian connection \(\nabla \)
  that satisfies \eqref{eqn8-7} and \eqref{eqn8-8}.
  For any \(Z,W\in \Gamma (H^{1,0}(M))\),
  we have
  \begin{align*}
    \Theta^{(1,1)}(Z,\overline{W})
    & = \pi_{+}(T_{\nabla }(Z,\overline{W})) \\
    & = \pi_{+}(\nabla_{Z}\overline{W}-\nabla_{\overline{W}}Z-[Z,\overline{W}])
    \\
    & = -\nabla_{\overline{W}}Z-\pi_{+}([Z,\overline{W}])
  \end{align*}
  that is,
  \begin{equation}
    \nabla_{\overline{W}}Z=\pi_{+}([\overline{W},Z])-\Psi (Z,\overline{W}).
    \label{eqn9-1}
  \end{equation}
  For any \(X,Y,Z\in H^{1,0}(M)\),
  we obtain from \(\nabla g_{h,\theta }=0\) and \eqref{eqn9-1} that
  \begin{equation*}
    X(g_{h,\theta }(Y,\overline{Z})
    =g_{h,\theta }(\nabla_{X}Y,\overline{Z})
    +g_{h,\theta }(Y,\pi_{-}([X,\overline{Z}]+\overline{\Psi (\overline{X},Z)})
  \end{equation*}
  that is,
  \begin{equation}
    h (\nabla_{X}Y,\overline{Z})
    = Xh(Y,\overline{Z})
    - h(Y,\pi_{-}([X,\overline{Z}])
    + \overline{\Psi (\overline{X},Z)}).
    \label{eqn9-2}
  \end{equation}
  Since \(h\) is a Hermitian metric on \(H(M)\),
  \eqref{eqn9-2} determines \(\nabla_{X}Y\) uniquely for any \(X,Y\in \Gamma (H^{1,0}(M))\).

  In terms of \eqref{eqn8-6} and Lemma \ref{lem:2},
  we obtain
  \begin{equation}
    \nabla_{\xi }X=\tau (X)+[\xi ,X]=\tau (X)+\mathcal{L}_{\xi }X  \label{eqn10}
  \end{equation}
  for any \(X\in \Gamma (TM)\),
  where \(\mathcal{L}\) denotes the Lie derivative.
  Since the pseudo-Hermitian connection satisfies \(\nabla J=0\),
  we deduce by means of \eqref{eqn8-6-1} and \eqref{eqn10} that
  \begin{align*}
    0 & = (\nabla_{\xi }J)(X) \\
      & = \nabla_{\xi }(JX)-J\nabla_{\xi }X \\
      & = \tau (JX)+\mathcal{L}_{\xi }JX-J(\tau (X)+\mathcal{L}_{\xi }X) \\
      & = -2J\tau (X)+\sigma (X)+(\mathcal{L}_{\xi }J)(X)
  \end{align*}
  for any \(X\in \Gamma (TM)\).
  Therefore
  \begin{equation}
    \tau (X)
    = -\frac{1}{2} \bigl[J\circ \sigma (X)
    +J\circ (\mathcal{L}_{\xi}J)(X) \bigr]  \label{eqn11}
  \end{equation}
  and
  \begin{equation}
    \nabla_{\xi }X
    =-\frac{1}{2}\bigl[J\circ \sigma (X)
    +J\circ (\mathcal{L}_{\xi }J)(X)\bigr]+\mathcal{L}_{\xi }X
    \label{eqn12}
  \end{equation}
  for any \(X\in \Gamma (TM)\).
  Since \(\nabla \) preserves the decomposition \eqref {eqn8},
  we see from \eqref{eqn8-6-3} and \eqref{eqn12} that
  \begin{align}
    \label{eqn12-1}
    \begin{aligned}
      \bigl( -\frac{1}{2} J\circ (\mathcal{L}_{\xi }J)+\mathcal{L}_{\xi }\bigr)
      (H^{1,0}(M)
      & \subseteq H^{1,0}(M) \\
      \\
      \bigl( -\frac{1}{2}J\circ (\mathcal{L}_{\xi }J)+\mathcal{L}_{\xi }\bigr)
      (H^{0,1}(M)
      & \subseteq H^{0,1}(M)
    \end{aligned}
  \end{align}
  Using \(\nabla g_{h,\theta }=0\),
  \eqref{eqn12} and \eqref{eqn8-8},
  we derive for any \(Z,W\in \Gamma (H^{1,0}(M))\) the following condition for \(\sigma \)
  \begin{align*}
    0
    & = (\nabla_{\xi }g_{h,\theta })(Z,\overline{W})
      =\xi g_{h,\theta }(Z, \overline{W})
      -g_{h,\theta }(\nabla_{\xi }Z,\overline{W})
      -g_{h,\theta}(Z,\nabla_{\xi }\overline{W}) \\
    & = \xi g_{h,\theta }(Z,\overline{W})
      -g_{h,\theta }(-\frac{1}{2}[J\circ \sigma (Z)+J\circ (\mathcal{L}_{\xi }J)(Z)]
      +\mathcal{L}_{\xi }Z,\overline{W})
    \\
    & \quad -g_{h,\theta }(Z,-\frac{1}{2}[J\circ \sigma (\overline{W})+J\circ (\mathcal{L}_{\xi }J)(\overline{W})]
      +\mathcal{L}_{\xi }\overline{W}) \\
    & = \xi h(Z,\overline{W})+\sqrt{-1}h(\sigma (Z),\overline{W})
      +h(\frac{1}{2} J\circ (\mathcal{L}_{\xi }J)(Z)-\mathcal{L}_{\xi }Z,\overline{W}) \\
    & \quad +h(Z,\frac{1}{2}J\circ (\mathcal{L}_{\xi }J)(\overline{W})-\mathcal{L}_{\xi }\overline{W})
  \end{align*}
  that is,
  \begin{align}
    \label{eqn12-2}
    \begin{aligned}
      h(\sigma (Z),\overline{W})
      & =\sqrt{-1}
        \biggl[ \xi h(Z,\overline{W})
        +h(\frac{1}{2}J\circ (\mathcal{L}_{\xi }J)(Z)-\mathcal{L}_{\xi }Z,\overline{W}) \\
      & \quad +h(Z,\frac{1}{2}J\circ (\mathcal{L}_{\xi }J)(\overline{W})-\mathcal{L}_{\xi } \overline{W}) \biggr]
    \end{aligned}
  \end{align}
  We conclude from \eqref{eqn9-1}, \eqref{eqn9-2}, \eqref{eqn12}, \eqref{eqn12-2} and Lemma \ref{lem:2} that the connection is uniqueness.

  Conversely, let
  \(\nabla :\Gamma (T(M)\otimes C)\times \Gamma (T(M)\otimes C)\rightarrow \Gamma (T(M)\otimes C)\)
  be a map defined by
  \begin{equation}
    \label{eqn13}
    \begin{gathered}
      \nabla_{\overline{X}}Y=\pi_{+}([\overline{X},Y])-\Psi (Y,\overline{X}), \qquad
      \nabla_{X}\overline{Y}=\overline{\nabla_{\overline{X}}Y},  \\
      \nabla_{X}Y=\beta_{X}Y, \qquad
      \nabla_{\overline{X}}\overline{Y}= \overline{\nabla_{X}Y}, \qquad
      \nabla \xi =0 \\
      \nabla_{\xi }X
      =-\frac{1}{2}[J\circ \sigma (X)+J\circ (\mathcal{L}_{\xi}J)(X)]+\mathcal{L}_{\xi }X, \qquad
      \nabla_{\xi }\overline{X}= \overline{\nabla_{\xi }X},
    \end{gathered}
  \end{equation}
  for any \(X,Y\in \Gamma (H^{1,0}(M))\).
  Here
  \begin{equation*}
    \beta :\Gamma (H^{1,0}(M))\times \Gamma (H^{1,0}(M))\rightarrow \Gamma (H^{1,0}(M))
  \end{equation*}
  is uniquely determined by
  \begin{equation}
    h(\beta_{X}Y,\overline{Z})
    =Xh(Y,\overline{Z})-h(Y,\pi_{-}([X,\overline{Z}])+\overline{\Psi (\overline{X},Z)})  \label{eqn14}
  \end{equation}
  for any \(X,Y,Z\in H^{1,0}(M)\),
  and \(\sigma \ \)is a unique \(TM\)-valued \(1\)-form satisfying \(\sigma (\xi )=0\),
  \(\sigma :\Gamma (H^{1,0}(M))\rightarrow \Gamma (H^{1,0}(M))\) and \eqref{eqn12-2}.
  Clearly the map \(\nabla \) defined by \eqref{eqn13} is a linear connection on \(T(M)\).
  The remaining part of the proof is to show that
  \(\nabla \) is a pseudo-Hermitian connection satisfying \( \Theta^{(1,1)}=\Psi \) and \eqref{eqn8-8}.

  It is easy to see that the covariant derivative \(\nabla_{X}\) defined by \eqref{eqn13} preserves the decomposition \eqref{eqn8} for any \(X\in \Gamma (H(M))\).
  Next,
  using \eqref{eqn1-2} and \eqref{eqn2-2},
  we find that
  \begin{equation}
    \mathcal{L}_{\xi }(\Gamma (H(M))\subseteq \Gamma (H(M))  \label{eqn15}
  \end{equation}
  and
  \begin{align}
    \label{eqn16}
    \begin{aligned}
      \mathcal{L}_{\xi }(J^{2})
      & = (\mathcal{L}_{\xi }J)\circ J+J\circ (\mathcal{L}_{\xi }J) \\
      & = \mathcal{L}_{\xi }(-I+\theta \otimes \xi ) \\
      & =  0.
    \end{aligned}
  \end{align}
  In view of \(\sigma (\xi )=0\) and \(\sigma (H^{1,0}(M))\subseteq H^{1,0}(M)\), 
  we get
  \begin{equation}
    \sigma \circ J=J\circ \sigma . \label{eqn16-1}
  \end{equation}
  According to \eqref{eqn13},
  \eqref{eqn15}, \eqref{eqn16} and \eqref{eqn16-1},
  we have
  \begin{equation}
    \nabla_{\xi }(\Gamma (H(M)))\subseteq \Gamma (H(M))  \label{eqn17}
  \end{equation}
  and
  \begin{align}
    \label{eqn18}
    \begin{aligned}
      \nabla_{\xi }(JX)
      & = -\frac{1}{2}[J\circ \sigma (JX)+J\circ (\mathcal{L}_{\xi }J)(JX)]+\mathcal{L}_{\xi }(JX)  \\
      & = \frac{1}{2}\sigma (X)-\frac{1}{2}(\mathcal{L}_{\xi }J)(X)
        +(\mathcal{L}_{\xi }J)(X)+J(\mathcal{L}_{\xi }X) \\
      & = \frac{1}{2}\sigma (X)+\frac{1}{2}(\mathcal{L}_{\xi }J)(X)+J(\mathcal{L}_{\xi }X)   \\
      & = J\biggl[ -\frac{1}{2}J\sigma (X)
        -\frac{1}{2}J\circ (\mathcal{L}_{\xi}J)(X)+(\mathcal{L}_{\xi }X)\biggr] \\
      & = J\nabla_{\xi }X
    \end{aligned}
  \end{align}
  for any \(X\in \Gamma (H)\).
  Then \eqref{eqn13},
  \eqref{eqn17}\ and \eqref {eqn18} show that \(\nabla_{\xi }\) preserves the decomposition \eqref{eqn8} too.
  We conclude from above that \(\nabla \) preserves the decomposition \eqref{eqn8}.

  In terms of \eqref{eqn13},
  we deduce
  \begin{align}
    \label{eqn19}
    \begin{aligned}
      T_{\nabla }(X,\overline{Y})
      & =\nabla_{X}\overline{Y}-\nabla_{\overline{Y}}X-[X,\overline{Y}]   \\
      & =\pi_{-}([X,\overline{Y}])+\overline{\Psi (\overline{X},Y)}+\pi_{+}([X, \overline{Y}])+\Psi (X,\overline{Y})-[X,\overline{Y}]   \\
      & =\Psi (X,\overline{Y})+\overline{\Psi (\overline{X},Y)}-\theta ([X, \overline{Y}])\xi \\
      & =\Psi (X,\overline{Y})+\overline{\Psi (\overline{X},Y)}+2d\theta (X, \overline{Y})\xi   \\
      & =\Psi (X,\overline{Y})+\overline{\Psi (\overline{X},Y)}+2\sqrt{-1} L_{\theta }(X,\overline{Y})\xi .
    \end{aligned}
  \end{align}
  Hence \(\Theta^{(1,1)}=\Psi \).

  By the definition of \(\tau \) and using \eqref{eqn13},
  we get
  \begin{align}
    \label{eqn20}
    \begin{aligned}
      \tau (X)
      & =T_{\nabla }(\xi ,X)=\nabla_{\xi }X-\nabla_{X}\xi -[\xi ,X]  \\
      & =-\frac{1}{2}[J\circ \sigma (X)+J\circ (\mathcal{L}_{\xi }J)(X)]
        +\mathcal{L}_{\xi }X-\mathcal{L}_{\xi }X   \\
      & =-\frac{1}{2}[J\circ \sigma (X)+J\circ (\mathcal{L}_{\xi }J)(X)].
    \end{aligned}
  \end{align}
  It follows from \eqref{eqn16}, \eqref{eqn16-1} and \eqref{eqn20} that
  \begin{align*}
    (J\circ \tau +\tau \circ J)X
    & =\frac{1}{2}\sigma (X)+\frac{1}{2}(\mathcal{L}_{\xi }J)(X)-\frac{1}{2}J\circ \sigma (JX) \\
    & \quad -\frac{1}{2}J\circ (\mathcal{L}_{\xi }J)(JX) \\
    & =\sigma (X)
  \end{align*}
  for any \(X\in \Gamma (H(M))\).
  Obviously \((J\circ \tau +\tau \circ J)\xi =\sigma (\xi )=0\).
  Thus \(\sigma \) is given by \eqref{eqn8-6-1}.
  Notice that \( \sigma \) is determined by\ \eqref{eqn12-2}.
  It is easy to verify by taking conjugation of \eqref{eqn12-2} that
  \begin{equation}
    h(\sigma (Z),\overline{W})=-h(Z,\sigma (\overline{W}))  \label{eqn21}
  \end{equation}
  for any \(Z,W\in H^{1,0}(M)\).
  This implies \eqref{eqn8-8}.

  Finally,
  we would like to prove \(\nabla g_{h,\theta }=0\).
  Obviously this is equivalent to \(\nabla \theta =0\) and \(\nabla h=0\).
  The former one follows from \(\theta (\xi )=1\),
  \(\nabla_{X}\xi =0\) and \(\nabla_{X}(\Gamma (H(M)))\subseteq \Gamma (H(M))\) for any \(X\in \Gamma (T(M))\).
  For the latter one,
  we first observe that
  \begin{align*}
    (\nabla_{X}h)(\xi ,Y)
    & =X(h(\xi ,Y))-h(\nabla_{X}\xi ,Y)-h(\xi ,\nabla_{X}Y) \\
    & =0
  \end{align*}
  for any \(X,Y\in \Gamma (T(M))\).
  Next,
  due to \eqref{eqn3-2},
  we have \( h(Z,W)=h(\overline{Z},\overline{W})=0\) for any \(Z,W\in H^{1,0}(M)\).
  Thus \( \nabla_{X}h\) vanishes on complex tangent vectors of the same type
  (i.e., both in \(H^{1,0}(M)\) or both in \(H^{0,1}(M)\)),
  because \(\nabla \) preserves the decomposition \eqref{eqn8}.
  Letting \(X,Y,Z\in \Gamma (H^{1,0}(M))\) and using \eqref{eqn14},
  we obtain
  \begin{align*}
    (\nabla_{X}h)(Y,\overline{Z})
    & =Xh(Y,\overline{Z})-h(\nabla_{X}Y, \overline{Z})-h(Y,\nabla_{X}\overline{Z}) \\
    & =Xh(Y,\overline{Z})-h(\beta_{X}Y,\overline{Z})
      -h(Y,\pi_{-}[X,\overline{Z}]+\overline{\Psi (\overline{X},Z)}) \\
    & =0.
  \end{align*}
  Consequently \(\nabla_{X}h=0\) and \(\nabla_{\overline{X}}h=0\) for any \(X\in H^{1,0}(M)\).
  Finally, \eqref{eqn12-2}, \eqref{eqn13} and \eqref{eqn21} imply
  \begin{align*}
    (\nabla_{\xi }h)(Y,\overline{Z})
    & =\xi h(Y,\overline{Z})-h(\nabla_{\xi }Y, \overline{Z})-h(Y,\nabla_{\xi }\overline{Z}) \\
    & =\xi h(Y,\overline{Z})-h\biggl( -\frac{1}{2}[J\sigma (Y)+J(\mathcal{L}_{\xi }J)(Y)]+\mathcal{L}_{\xi }Y,\overline{Z} \biggr) \\
    & \quad -h\biggl( Y,-\frac{1}{2}[J\sigma (\overline{Z})+J(\mathcal{L}_{\xi }J)( \overline{Z})]+\mathcal{L}_{\xi }\overline{Z} \biggr) \\
    & =\xi h(Y,\overline{Z})+\sqrt{-1}h(\sigma (Y),\overline{Z})+h \biggl(\frac{1}{2}J(\mathcal{L}_{\xi }J)(Y)-\mathcal{L}_{\xi }Y,\overline{Z} \biggr) \\
    & \quad +h\biggl( Y,\frac{1}{2}J(\mathcal{L}_{\xi }J)(\overline{Z})-\mathcal{L}_{\xi }\overline{Z} \biggr) \\
    & =0
  \end{align*}
  for any \(Y,Z\in \Gamma (H^{1,0}(M))\).
  We conclude that \(\nabla g_{h,\theta}=0\).
  This complete the proof of Theorem \ref{thm:5}.
\end{proof}

\begin{remark}
  \label{rmk:1}                 
  Even if \(M\) is a degenerate CR manifold,
  we may also use \eqref{eqn13} to define a pseudo-Hermitian connection satisfying \eqref{eqn8-7} and \eqref {eqn8-8},
  provided that \(T(M)\otimes \mathbb{C}\) admits the decomposition \eqref{eqn8} and \(\mathcal{L}_{\xi }\theta =0\).
  In particular,
  this includes the case that \((M,\theta ,J,h)\) is a degenerate Hopf real hypersurface in a Hermitian manifold,
  which has the induced structures from the ambient space.
\end{remark}

\begin{corollary}
  \label{cor:3}                 
  Let \((M,\theta ,J,h)\) be a pseudo-Hermitian manifold.
  Then there is a unique pseudo-Hermitian connection \(\nabla \) such that \(\Theta^{(1,1)}=0\) and \(h(\sigma (X),Y)=-h(X,\sigma (Y))\) for any \(X,Y\in T(M)\).
\end{corollary}

The connection in Corollary \ref{cor:3} will be called \emph{the canonical pseudo-Hermitian connection} of \((M,\theta ,J,h)\).
It may be regarded as a generalization of the Chern connection on Hermitian manifolds.
Henceforth we take this connection as a reference connection.
Using Lemma \ref{lem:1} (i),
\eqref {eqn8-4} and \eqref{eqn8-6},
it is easy to see that the total torsion of the canonical connection can be expressed as
\begin{equation}
  \label{eqn22}
  T_{\nabla }(\cdot ,\cdot )
  =2\theta \wedge \tau +\Theta^{(2,0)}
  +\overline{\Theta^{(2,0)}}+2d\theta (\cdot ,\cdot )\xi .
\end{equation}
The vanishing torsion components of the canonical connection reflect the level of integrability of the \(\{1\}\times U(m)\)-structure on \(M^{2m+1}\).
In view of the above results,
the remaining torsion components \(\Theta^{(2,0)}\) and \(\tau \) of the canonical connection should contain essential information about the pseudo-Hermitian manifold.
We will investigate pseudo-Hermitian manifolds with further restrictions or properties on these components later.
It should be pointed out that if \((M,\theta ,J)\) is a strictly pseudoconvex CR manifold,
then the connection defined by \eqref{eqn13} with \( h=L_{\theta }\) and \(\Psi =0\) coincides with the Tanaka-Webster connection (see also Remark \ref{rmk:2}).

Although we are mainly interested in endowing \(H(M)\) with positive definite metrics in this paper,
we would like to end this section by the following remark: if \(\widetilde{h}\) is a fiberwise pseudo-Riemannian metric on \(H(M)\) satisfying \eqref{eqn3-2},
the argument for Theorem \ref{thm:5} also gives us a unique "canonical connection" compatible with \((M,\theta ,J,\widetilde{h})\).
This includes the case that \(\widetilde{h}=L_{\theta }\) when \((M,\theta ,J)\) is only a non-degenerate CR manifold.

\section{Structure equations and Bianchi identities} \label{sec:struct-equat-bianchi}

In this section,
we derive the structure equations for the canonical connection on a pseudo-Hermitian manifold.
Let \((M^{2m+1},\theta ,J,h)\) be a pseudo-Hermitian manifold and let \(\nabla \) be its canonical connection.
For simplicity, 
we will sometimes write \(g_{h,\theta }\) as \(\langle \cdot ,\cdot \rangle\) in the following. 
Let us choose a local orthonormal frame field
\(\{e_{A}\}_{A=0}^{2m}=\{\xi ,e_{1},...,e_{m},e_{m+1},...,e_{2m}\}\) with respect to
\(g_{h,\theta }=\langle \cdot ,\cdot \rangle\) such that
\begin{equation*}
  \{e_{m+1}, \cdots, e_{2m}\}=\{Je_{1}, \cdots, Je_{m}\}.
\end{equation*}
Set
\begin{equation}
  \eta_{0}=\xi,
  \eta_{\alpha }=\frac{1}{\sqrt{2}}(e_{\alpha }-\sqrt{-1}Je_{\alpha }),
  \eta_{\overline{\alpha }}=\frac{1}{\sqrt{2}} (e_{\alpha }+\sqrt{-1}Je_{\alpha }) \quad
  (\alpha =1, \cdots, m).
  \label{3.1}
\end{equation}
Let
\(\{\theta ,\theta^{\alpha },\theta^{\overline{\alpha }}\}\)
be the dual frame field of \(\{\eta_{0}=\xi ,\eta_{\alpha },\eta_{\overline{\alpha }}\} \).
Henceforth we shall make use of the following convention on the ranges of indices:
\begin{gather*}
  A,B,C, \cdots = 0,1, \cdots, m,\overline{1}, \cdots, \overline{m}, \\
  \alpha ,\beta ,\gamma , \cdots = 1, \cdots, m, \quad
  \overline{\alpha }, \overline{\beta },\overline{\gamma }, \cdots
  =\overline{1}, \cdots, \overline{m}.
\end{gather*}
As usual repeated indices are summed over the respective ranges.

Writing
\begin{equation}
  T_{\nabla }(\cdot ,\cdot )
  =T^{0}(\cdot ,\cdot )\xi +T^{\alpha }(\cdot ,\cdot)\eta_{\alpha }
  +T^{\overline{\alpha }}(\cdot ,\cdot )\eta_{\overline{\alpha }},  \label{3.1-0}
\end{equation}
we have
\begin{equation}
  \Theta =\pi_{+}(T_{\nabla }((\cdot ,\cdot ))
  = T^{\alpha }(\cdot ,\cdot )\eta_{\alpha }  \label{3.2}
\end{equation}
in view of \eqref{eqn8-4}.
Recall that \(\Theta^{(0,2)}=\Theta^{(1,1)}=0\) and \(\tau \) is \(H(M)\)-valued for the canonical connection.
In terms of the dual frame field,
we may express \(\tau \) and \(\Theta^{(2,0)}\) as follows
\begin{equation}
  \tau =\tau^{\alpha }\eta_{\alpha }+\tau^{\overline{\alpha }}\eta_{\overline{\alpha }}
  \label{3.3}
\end{equation}
with
\begin{equation}
  \label{3.9-1}
  \begin{cases}
    \tau^{\alpha }
    =A_{\beta }^{\alpha }\theta^{\beta }
    +A_{\overline{\beta }}^{\alpha }\theta^{\overline{\beta }} \\
    \tau^{\overline{\alpha }}
    =A_{\beta }^{\overline{\alpha }}\theta^{\beta}
    +A_{\overline{\beta }}^{\overline{\alpha }}\theta^{\overline{\beta }}
  \end{cases}
\end{equation}
and
\begin{equation}
  \Theta^{(2,0)}
  =(T_{\beta \gamma }^{\alpha }\theta^{\beta }\wedge \theta^{\gamma })\eta_{\alpha }  \label{3.4}
\end{equation}
where \(T_{\beta \gamma }^{\alpha }=T^{\alpha }(\eta_{\beta },\eta_{\gamma})\).
From \eqref{eqn22}, \eqref{3.2}, \eqref{3.9-1} and \eqref{3.4},
we obtain
\begin{align}
  \label{3.5}
  \begin{aligned}
    T^{\alpha }
    & = 2\theta \wedge \tau^{\alpha }
      +T_{\beta \gamma }^{\alpha}\theta^{\beta }\wedge \theta^{\gamma }  \\
    & = 2A_{\beta }^{\alpha }\theta \wedge \theta^{\beta }
      +2A_{\overline{\beta }}^{\alpha }\theta \wedge \theta^{\overline{\beta }}
      +T_{\beta \gamma}^{\alpha }\theta^{\beta }\wedge \theta^{\gamma }
  \end{aligned}
\end{align}
It follows from \eqref{3.3} and \eqref{3.9-1} that
\begin{align}
  \label{3.9-2}
  \begin{aligned}
    \sigma
    & = \tau \circ J+J\circ \tau   \\
    & = \sqrt{-1}(A_{\mu }^{\lambda }\theta^{\mu }-A_{\overline{\mu }}^{\lambda}\theta^{\overline{\mu }})\eta_{\lambda }
      +\sqrt{-1}(A_{\mu }^{\overline{\lambda }}\theta^{\mu }-A_{\overline{\mu }}^{\overline{\lambda }}\theta^{\overline{\mu }})\eta_{\overline{\lambda }}   \\
    & \quad +\sqrt{-1}(A_{\mu }^{\lambda }\theta^{\mu }+A_{\overline{\mu }}^{\lambda}\theta^{\overline{\mu }})\eta_{\lambda }
      -\sqrt{-1}(A_{\mu }^{\overline{\lambda }}\theta^{\mu }+A_{\overline{\mu }}^{\overline{\lambda }}\theta^{\overline{\mu }})\eta_{\overline{\lambda }}   \\
    & = 2\sqrt{-1}\biggl[ (A_{\mu }^{\lambda }\theta^{\mu })\eta_{\lambda }-(A_{\overline{\mu }}^{\overline{\lambda }}\theta^{\overline{\mu }})\eta_{\overline{\lambda }}\biggr] .
  \end{aligned}
\end{align}
Using \eqref{eqn8-8} and \eqref{3.9-2},
we obtain
\begin{equation*}
  2\sqrt{-1}A_{\beta }^{\alpha }
  =\langle \sigma (\eta_{\beta }),\eta_{\overline{\alpha }}\rangle
  =-\langle \eta_{\beta },\sigma (\eta_{\overline{\alpha }})\rangle
  =2\sqrt{-1}A_{\overline{\alpha }}^{\overline{\beta }}
\end{equation*}
that is,
\begin{equation}
  A_{\beta }^{\alpha }=A_{\overline{\alpha }}^{\overline{\beta }}.
  \label{3.9-2-1}
\end{equation}
Note that \eqref{3.9-2-1} means
\(\langle \tau (\eta_{\beta }),\eta_{\overline{\alpha }}\rangle
=\langle \eta_{\beta },\tau (\eta_{\overline{\alpha }})\rangle \).

According to the properties of the canonical connection, we have
\begin{equation*}
  \nabla_{X}\eta_{0}=0, \quad
  \nabla_{X}\eta_{\beta }=\theta_{\beta}^{\alpha }(X)\eta_{\alpha }, \quad
  \nabla_{X}\eta_{\overline{\beta }}
  =\theta_{\overline{\beta }}^{\overline{\alpha }}(X)\eta_{\overline{\alpha}}
\end{equation*}
for any \(X\in TM\), where
\(\{\theta_{0}^{0}=\theta_{0}^{\alpha }=\theta_{0}^{\overline{\alpha }}
=\theta_{\alpha }^{0}=\theta_{\overline{\alpha }}^{0}=0,
\theta_{\beta }^{\alpha },
\theta_{\overline{\beta }}^{\overline{\alpha }}\}\)
are the connection \(1\)-forms of \(\nabla \) with respect to
\(\{\eta_{0}=\xi ,\eta_{\beta },\eta_{\overline{\beta }}\}\).

\begin{lemma}[The first structure equations]
  \label{lem:3}                 
  Under the above notations, we have
  \begin{align}
    \label{3.6}
    \begin{aligned}
      d\theta
      & = 2\sqrt{-1}L_{\alpha \overline{\beta }}\theta^{\alpha }\wedge \theta^{\overline{\beta }}   \\
      d\theta^{\alpha }
      & = -\theta_{\beta }^{\alpha }\wedge \theta^{\beta }+\frac{1}{2}T^{\alpha }
        = -\theta_{\beta }^{\alpha }\wedge \theta^{\beta}
        +\theta \wedge \tau^{\alpha }
        + \frac{1}{2}T_{\beta \gamma }^{\alpha }\theta^{\beta }\wedge \theta^{\gamma } \\
      0
      & = \theta_{\beta }^{\alpha }+\theta_{\overline{\alpha }}^{\overline{\beta }}
    \end{aligned}
  \end{align}
  where \(L_{\alpha \overline{\beta }}=L_{\theta }(\eta_{\alpha },\eta_{\overline{\beta }})\).
\end{lemma}

\begin{proof}
  For any \(X,Y\in \Gamma (TM)\),
  we compute
  \begin{align}
    T_{\nabla }(X,Y)
    & = \nabla_{X}
      \biggl\{\theta (Y)\xi +\theta^{\alpha}(Y)\eta_{\alpha }
      +\theta^{\overline{\alpha }}(Y)\eta_{\overline{\alpha }} \biggr\}  \nonumber \\
    & \quad - \nabla_{Y}\biggl\{\theta (X)\xi +\theta^{\alpha }(X)\eta_{\alpha}
      +\theta^{\overline{\alpha }}(X)\eta_{\overline{\alpha }}\biggr\} \nonumber \\
    & \quad -\biggl\{\theta ([X,Y])\xi +\theta^{\alpha }([X,Y])\eta_{\alpha}
      +\theta^{\overline{\alpha }}([X,Y])\eta_{\overline{\alpha }}\biggr\} \nonumber \\
    & = 2d\theta (X,Y)\xi +2d\theta^{\alpha }(X,Y)\eta_{\alpha }
      +2d\theta^{\overline{\alpha }}(X,Y)\eta_{\overline{\alpha }}  \nonumber \\
    & \quad +2(\theta_{\beta }^{\alpha }\wedge \theta^{\beta })(X,Y)\eta_{\alpha}
      +2(\theta_{\overline{\beta }}^{\overline{\alpha }}
      \wedge \theta^{\overline{\beta }})(X,Y)\eta_{\overline{\alpha }}  \label{3.7}
  \end{align}
  It follows from \eqref{3.1-0} and \eqref{3.7} that
  \begin{equation}
    d\theta =\frac{1}{2}T^{0}, \quad
    d\theta^{\alpha }
    = -\theta_{\beta}^{\alpha }\wedge \theta^{\beta }+\frac{1}{2}T^{\alpha }  \label{3.8}
  \end{equation}
  From Lemmas \ref{lem:1}, \ref{lem:2} and \eqref{3.5}, \eqref{3.8}, 
  we have the first two equations of \eqref{3.6}.
  Taking derivative of
  \(\langle \eta_{\alpha },\eta_{\overline{\beta }}\rangle =\delta_{\alpha \overline{\beta }}\)
  with respect to any \(X\in TM\), we get
  \begin{equation*}
    0=X\langle \eta_{\alpha },\eta_{\overline{\beta }}\rangle
    =\theta_{\alpha}^{\beta }(X)+\theta_{\overline{\beta }}^{\overline{\alpha }}(X)=0.
  \end{equation*}
  This gives the third equation in \eqref{3.6}.
\end{proof}

The curvature tensor of \(\nabla \) is defined by
\begin{equation*}
  R(X,Y)Z=\nabla_{X}\nabla_{Y}Z-\nabla_{Y}\nabla_{X}Z-\nabla_{[X,Y]}Z
\end{equation*}
whose components \(\{R_{BCD}^{A}\}\) with respect to the frame fields \(\{\eta_{A}\}\) are given by
\begin{equation}
  R(\eta_{C},\eta_{D})\eta_{B}=R_{BCD}^{A}\eta_{A}
  =R_{BCD}^{0}\eta_{0}+R_{BCD}^{\alpha }\eta_{\alpha }
  +R_{BCD}^{\overline{\alpha }}\eta_{\overline{\alpha }}
  \label{3.9-3-1}
\end{equation}
with
\begin{equation}
  R_{\overline{\beta }CD}^{\alpha}
  = R_{\beta CD}^{\overline{\alpha }}
  = R_{BCD}^{0}=R_{0CD}^{A}=0, \quad
  R_{BCD}^{A}=-R_{BDC}^{A}.
  \label{3.9.3-2}
\end{equation}
Set
\begin{equation*}
  R_{\overline{\alpha }BCD}=R_{BCD}^{\alpha }
  =\langle R(\eta_{C},\eta_{D})\eta_{B},\eta_{\overline{\alpha }}\rangle
\end{equation*}
and \(R_{\alpha BCD}=R_{BCD}^{\overline{\alpha }}\).
The curvature forms of \(\nabla \) are defined by
\begin{equation}
  R(\cdot ,\cdot )\eta_{\beta }
  =\Omega_{\beta }^{\alpha }(\cdot ,\cdot )\eta_{\alpha }, \quad
  R(\cdot ,\cdot )\eta_{\overline{\beta }}
  =\Omega_{\overline{\beta }}^{\overline{\alpha }}(\cdot ,\cdot )\eta_{\overline{\alpha }}.
  \label{3.9-3-2}
\end{equation}

\begin{lemma}[The second structure equations]
  \label{lem:4}                 
  Under the above notations, we have
  \begin{align}
    d\theta_{\beta }^{\alpha }
    & = -\theta_{\gamma }^{\alpha }\wedge \theta_{\beta }^{\gamma }
      +\frac{1}{2}\Omega_{\beta }^{\alpha }  \nonumber \\
    & = -\theta_{\gamma }^{\alpha }\wedge \theta_{\beta }^{\gamma }
      +\frac{1}{2}(R_{\beta \lambda \mu }^{\alpha }\theta^{\lambda }\wedge \theta^{\mu}
      +R_{\beta \overline{\lambda }\overline{\mu }}^{\alpha }\theta^{\overline{\lambda }}
      \wedge \theta^{\overline{\mu }}) \nonumber \\
    & \quad + R_{\beta \lambda \overline{\mu }}^{\alpha }\theta^{\lambda }
      \wedge \theta^{\overline{\mu }}
      +R_{\beta 0\mu }^{\alpha }\theta \wedge \theta^{\mu }
      +R_{\beta 0\overline{\mu }}^{\alpha }\theta \wedge \theta^{\overline{\mu }}  \label{3.9-3-3}
  \end{align}
  with
  \begin{align}
    \label{3.9.3-4}
    \begin{gathered}
      R_{\beta \lambda \mu }^{\alpha }
      = -R_{\beta \mu \lambda }^{\alpha }, \quad
      R_{\beta \overline{\lambda }\overline{\mu }}^{\alpha }
      =-R_{\beta \overline{\mu }\overline{\lambda }}^{\alpha }, \quad
      R_{\beta \lambda \overline{\mu }}^{\alpha }=-R_{\beta \overline{\mu }\lambda }^{\alpha }, \\
      R_{\beta 0\mu }^{\alpha }
      = -R_{\beta \mu 0}^{\alpha }, \quad
      R_{\beta 0\overline{\mu }}^{\alpha }
      =-R_{\beta \overline{\mu }0}^{\alpha }, \\
      R_{\overline{\alpha }BCD} = -R_{B\overline{\alpha }CD}, \quad
      R_{\alpha BCD}=-R_{B\alpha CD}.
    \end{gathered}
  \end{align}
\end{lemma}

\begin{proof}
  A direct computation yields
  \begin{align}
    R(X,Y)\eta_{\beta }
    & = \nabla_{X}\nabla_{Y}\eta_{\beta }-\nabla_{Y}\nabla_{X}\eta_{\beta }
      -\nabla_{\lbrack X,Y]}\eta_{\beta }  \nonumber \\
    & = \nabla_{X}\left( \theta_{\beta }^{\alpha }(Y)\eta_{\alpha }\right)
      -\nabla_{Y}\left( \theta_{\beta }^{\alpha }(X)\eta_{\alpha }\right)
      -\theta_{\beta }^{\alpha }([X,Y])\eta_{\alpha }  \nonumber \\
    & = 2d\theta_{\beta }^{\alpha }(X,Y)\eta_{\alpha }
      +\theta_{\beta }^{\alpha}(Y)\theta_{\alpha }^{\gamma }(X)\eta_{\gamma }
      -\theta_{\beta }^{\alpha}(X)\theta_{\alpha }^{\gamma }(Y)\eta_{\gamma }  \nonumber \\
    & = 2\biggl[d\theta_{\beta }^{\alpha }(X,Y)
      +(\theta_{\gamma }^{\alpha}\wedge \theta_{\beta }^{\gamma })\biggr](X,Y)\eta_{\alpha }.  \label{3.10}
  \end{align}
  In view of \eqref{3.9-3-1}, \eqref{3.9-3-2} and \eqref{3.10}, we have
  \begin{align*}
    \Omega_{\beta }^{\alpha }
    & = R_{\beta \lambda \mu }^{\alpha }\theta^{\lambda }\wedge \theta^{\mu }
      +R_{\beta \overline{\lambda }\overline{\mu }}^{\alpha }\theta^{\overline{\lambda }}
      \wedge \theta^{\overline{\mu }} \\
    & \quad +2(R_{\beta \lambda \overline{\mu }}^{\alpha }\theta^{\lambda }
      \wedge \theta^{\overline{\mu }}
      +R_{\beta 0\mu }^{\alpha }\theta \wedge \theta^{\mu }
      +R_{\beta 0\overline{\mu }}^{\alpha }\theta \wedge \theta^{\overline{\mu }})
  \end{align*}
  and
  \begin{align*}
    d\theta_{\beta }^{\alpha }
    & = -\theta_{\gamma }^{\alpha }\wedge \theta_{\beta }^{\gamma }
      +\frac{1}{2}\Omega_{\beta }^{\alpha } \\
    & = -\theta_{\gamma }^{\alpha }\wedge \theta_{\beta }^{\gamma }
      +\frac{1}{2}(R_{\beta \lambda \mu }^{\alpha }\theta^{\lambda }\wedge \theta^{\mu}
      +R_{\beta \overline{\lambda }\overline{\mu }}^{\alpha }\theta^{\overline{\lambda }}
      \wedge \theta^{\overline{\mu }}) \\
    & \quad
      +R_{\beta \lambda \overline{\mu }}^{\alpha }\theta^{\lambda }\wedge \theta^{\overline{\mu }}
      +R_{\beta 0\mu }^{\alpha }\theta \wedge \theta^{\mu }
      +R_{\beta 0\overline{\mu }}^{\alpha }\theta \wedge \theta^{\overline{\mu }}.
  \end{align*}
  Since \(\nabla g_{h,\theta }=0\),
  \begin{equation*}
    R_{\overline{\alpha }BCD}=-R_{B\overline{\alpha }CD}, \quad
    R_{\alpha BCD}=-R_{B\alpha CD}.
  \end{equation*}
  The other properties in \eqref{3.9.3-4} are clear from \eqref{3.9.3-2}.
\end{proof}

We are now going to derive the first Bianchi identity from the structure equations.
By taking the exterior derivative of the second equation of \eqref {3.6}
and using the structure equations, we have
\begin{align*}
  0
  & = -\left( d\theta_{\beta }^{\alpha }\wedge \theta^{\beta }-\theta_{\beta }^{\alpha }\wedge d\theta^{\beta }\right)
    +\frac{1}{2}dT^{\alpha } \\
  & = -\biggl( -\theta_{\gamma }^{\alpha }\wedge \theta_{\beta }^{\gamma }
    +\frac{1}{2}\Omega_{\beta }^{\alpha }\biggr) \wedge \theta^{\beta }
    +\theta_{\beta }^{\alpha }\wedge \left( -\theta_{\gamma }^{\beta }\wedge \theta^{\gamma }
    +\frac{1}{2}T^{\beta }\right)
    +\frac{1}{2}dT^{\alpha } \\
  & = \theta_{\gamma }^{\alpha }\wedge \theta_{\beta }^{\gamma }\wedge \theta^{\beta }
    -\frac{1}{2}\Omega_{\beta }^{\alpha }\wedge \theta^{\beta}
    -\theta_{\beta }^{\alpha }\wedge \theta_{\gamma }^{\beta }\wedge \theta^{\gamma }
    +\frac{1}{2}\theta_{\beta }^{\alpha }\wedge T^{\beta }
    +\frac{1}{2}dT^{\alpha } \\
  & = -\frac{1}{2}\Omega_{\beta }^{\alpha }\wedge \theta^{\beta }
    +\frac{1}{2}\theta_{\beta }^{\alpha }\wedge T^{\beta }
    +\frac{1}{2}dT^{\alpha }
\end{align*}
that is,
\begin{equation}
  \frac{1}{2}\Omega_{\beta }^{\alpha }\wedge \theta^{\beta }
  =\frac{1}{2} \{dT^{\alpha }+T^{\beta }\wedge \theta_{\beta }^{\alpha }\}.  \label{3.11}
\end{equation}
Using \eqref{3.5}, we deduce from \eqref{3.11} the following
\begin{align}
  \label{3.12}
  \begin{aligned}
    \frac{1}{2}\Omega_{\beta }^{\alpha }\wedge \theta^{\beta }
    & = DA_{\beta}^{\alpha }\wedge \theta \wedge \theta^{\beta }
      + DA_{\overline{\beta }}^{\alpha }\wedge \theta \wedge \theta^{\overline{\beta }}
      + \frac{1}{2}DT_{\beta \gamma }^{\alpha }\wedge \theta^{\beta }\wedge \theta^{\gamma }  \\
    & \quad
      + 2\sqrt{-1}A_{\beta }^{\alpha }L_{\lambda \overline{\mu }}\theta^{\lambda}
      \wedge \theta^{\overline{\mu }}\wedge \theta^{\beta }
      + 2\sqrt{-1}A_{\overline{\beta }}^{\alpha }L_{\lambda \overline{\mu }}\theta^{\lambda}
      \wedge \theta^{\overline{\mu }}\wedge \theta^{\overline{\beta }} \\
    & \quad -\frac{1}{2}A_{\beta }^{\alpha }T_{\lambda \mu }^{\beta }\theta
      \wedge \theta^{\lambda }\wedge \theta^{\mu }
      -\frac{1}{2}A_{\overline{\beta }}^{\alpha }
      T_{\overline{\lambda }\overline{\mu }}^{\overline{\beta }}\theta
      \wedge \theta^{\overline{\lambda }}\wedge \theta^{\overline{\mu }}
      + \frac{1}{2}T_{\beta \gamma }^{\alpha }T^{\beta }\wedge \theta^{\gamma }
  \end{aligned}
\end{align}
where
\begin{align}
  \label{3.13}
  \begin{aligned}
    DA_{\beta }^{\alpha }
    & = dA_{\beta }^{\alpha }-A_{\gamma }^{\alpha }\theta_{\beta }^{\gamma }
      + A_{\beta }^{\gamma }\theta_{\gamma }^{\alpha }=A_{\beta ,0}^{\alpha }\theta
      + A_{\beta ,\gamma }^{\alpha }\theta^{\gamma }
      + A_{\beta ,\overline{\gamma }}^{\alpha }\theta^{\overline{\gamma }}   \\
    DA_{\overline{\beta }}^{\alpha }
    & = dA_{\overline{\beta }}^{\alpha }
      -A_{\overline{\gamma }}^{\alpha }\theta_{\overline{\beta }}^{\overline{\gamma }}
      + A_{\overline{\beta }}^{\gamma }\theta_{\gamma }^{\alpha }=A_{\overline{\beta },0}^{\alpha }\theta
      + A_{\overline{\beta },\gamma }^{\alpha }\theta^{\gamma }
      + A_{\overline{\beta },\overline{\gamma }}^{\alpha }\theta^{\overline{\gamma }} \\
    DT_{\beta \gamma }^{\alpha }
    & = dT_{\beta \gamma }^{\alpha }-T_{\lambda \gamma }^{\alpha }\theta_{\beta }^{\lambda }
      -T_{\beta \lambda }^{\alpha}\theta_{\gamma }^{\lambda }
      + T_{\beta \gamma }^{\lambda }\theta_{\lambda}^{\alpha }=T_{\beta \gamma ,0}^{\alpha }\theta
      + T_{\beta \gamma ,\lambda}^{\alpha }\theta^{\lambda }
      + T_{\beta \gamma ,\overline{\lambda }}^{\alpha}\theta^{\overline{\lambda }}.
  \end{aligned}
\end{align}
Comparing the forms of the same types on both sides of \eqref{3.12}, we obtain
\begin{equation}
  R_{\beta \lambda \mu }^{\alpha }\theta^{\beta }\wedge \theta^{\lambda}\wedge \theta^{\mu }
  =T_{\beta \lambda ,\mu }^{\alpha }\theta^{\beta}\wedge \theta^{\lambda }\wedge \theta^{\mu }
  + T_{\delta \beta }^{\alpha}T_{\lambda \mu }^{\delta }\theta^{\beta }\wedge
  \theta^{\lambda }\wedge \theta^{\mu }  \label{3.14}
\end{equation}
\begin{equation}
  R_{\beta \overline{\lambda }\overline{\mu }}^{\alpha }\theta^{\beta }\wedge
  \theta^{\overline{\lambda }}\wedge \theta^{\overline{\mu }}
  = 4\sqrt{-1}L_{\beta \overline{\lambda }}A_{\overline{\mu }}^{\alpha }\theta^{\beta}\wedge
  \theta^{\overline{\lambda }}\wedge \theta^{\overline{\mu }} \label{3.15}
\end{equation}
\begin{equation}
  R_{\beta \lambda \overline{\mu }}^{\alpha }\theta^{\beta }\wedge
  \theta^{\lambda }\wedge \theta^{\overline{\mu }}
  =\frac{1}{2}T_{\beta \lambda ,\overline{\mu }}^{\alpha }\theta^{\beta }\wedge
  \theta^{\lambda }\wedge \theta^{\overline{\mu }}
  + 2\sqrt{-1}A_{\beta }^{\alpha }L_{\lambda \overline{\mu }}\theta^{\beta }\wedge
  \theta^{\lambda }\wedge \theta^{\overline{\mu}}  \label{3.16}
\end{equation}
\begin{align}
  R_{\beta 0\mu }^{\alpha }\theta^{\beta }\wedge \theta \wedge \theta^{\mu }
  & = -A_{\beta ,\mu }^{\alpha }\theta^{\beta }\wedge \theta \wedge \theta^{\mu }
    -\frac{1}{2}T_{\beta \mu ,0}^{\alpha }\theta^{\beta }\wedge \theta \wedge \theta^{\mu }  \nonumber \\
  & \quad
    + \frac{1}{2}A_{\lambda }^{\alpha }T_{\beta \mu }^{\lambda }\theta^{\beta}\wedge
    \theta \wedge \theta^{\mu }-T_{\beta \lambda }^{\alpha }A_{\mu}^{\lambda }\theta^{\beta }
    \wedge \theta \wedge \theta^{\mu }  \label{3.17}
\end{align}
\begin{align}
  R_{\beta 0\overline{\mu }}^{\alpha }\theta^{\beta }\wedge \theta \wedge \theta^{\overline{\mu }}
  & = -A_{\beta ,\overline{\mu }}^{\alpha }\theta^{\beta }\wedge \theta \wedge \theta^{\overline{\mu }}
    + A_{\overline{\mu },\beta }^{\alpha }\theta^{\beta }\wedge \theta \wedge \theta^{\overline{\mu }}
    \nonumber \\
  & \quad
    + T_{\lambda \beta }^{\alpha }A_{\overline{\mu }}^{\lambda }\theta^{\beta}\wedge \theta \wedge
    \theta^{\overline{\mu }}  \label{3.18}
\end{align}
\begin{equation}
  A_{\overline{\beta },\overline{\gamma }}^{\alpha }\theta^{\overline{\beta }}\wedge
  \theta \wedge \theta^{\overline{\gamma }}
  =\frac{1}{2}A_{\overline{\lambda }}^{\alpha }
  T_{\overline{\beta }\overline{\gamma }}^{\overline{\lambda }}
  \theta^{\overline{\beta }}\wedge \theta \wedge \theta^{\overline{\gamma }}  \label{3.19}
\end{equation}
From \eqref{3.15}, we get
\begin{equation}
  R_{\beta \overline{\lambda }\overline{\mu }}^{\alpha }
  =2\sqrt{-1}(L_{\beta \overline{\lambda }}A_{\overline{\mu }}^{\alpha }
  -L_{\beta \overline{\mu }}A_{\overline{\lambda }}^{\alpha })  \label{3.20}
\end{equation}
Using \eqref{3.9.3-4} and \eqref{3.20}, we have
\begin{align}
  R_{\beta \lambda \mu }^{\alpha }
  & = R_{\overline{\alpha }\beta \lambda \mu}=-R_{\beta \overline{\alpha }\lambda \mu }  \nonumber \\
  & = -\overline{R_{\overline{\beta }\alpha \overline{\lambda }\overline{\mu }}}
    = -\overline{R_{\alpha \overline{\lambda }\overline{\mu }}^{\beta }}  \nonumber \\
  & = -\overline{2\sqrt{-1}(L_{\alpha \overline{\lambda }}A_{\overline{\mu }}^{\beta }
    -L_{\alpha \overline{\mu }}A_{\overline{\lambda }}^{\beta })} \nonumber \\
  & = 2\sqrt{-1} ( L_{\overline{\alpha }\lambda }A_{\mu }^{\overline{\beta }}
    -L_{\overline{\alpha }\mu }A_{\lambda }^{\overline{\beta }} ) \label{3.21}
\end{align}
Clearly \eqref{3.16} implies
\begin{equation}
  R_{\beta \lambda \overline{\mu }}^{\alpha }-R_{\lambda \beta \overline{\mu }}^{\alpha }
  =T_{\beta \lambda ,\overline{\mu }}^{\alpha }
  + 2\sqrt{-1}(A_{\beta}^{\alpha }L_{\lambda \overline{\mu }}
  -A_{\lambda }^{\alpha }L_{\beta \overline{\mu }}).  \label{3.22}
\end{equation}
It follows from \eqref{3.18} that
\begin{equation}
  R_{\beta 0\overline{\mu }}^{\alpha }
  =A_{\overline{\mu },\beta }^{\alpha}-A_{\beta ,\overline{\mu }}^{\alpha }
  -T_{\beta \lambda }^{\alpha }A_{\overline{\mu }}^{\lambda }  \label{3.23}
\end{equation}
Consequently, using \eqref{3.9.3-4} again, we deduce from \eqref{3.23} that
\begin{align}
  R_{\beta 0\mu }^{\alpha }
  & = R_{\overline{\alpha }\beta 0\mu }
    =-R_{\beta \overline{\alpha }0\mu }  \nonumber \\
  & = -\overline{R_{\overline{\beta }\alpha 0\overline{\mu }}}
    =-\overline{R_{\alpha 0\overline{\mu }}^{\beta }}
    \nonumber \\
  & = -\overline{(A_{\overline{\mu },\alpha }^{\beta }-A_{\alpha ,\overline{\mu}}^{\beta }
    -T_{\alpha \lambda }^{\beta }A_{\overline{\mu }}^{\lambda })} \nonumber \\
  & = A_{\overline{\alpha },\mu }^{\overline{\beta }}-A_{\mu ,\overline{\alpha }}^{\overline{\beta }}
    + T_{\overline{\alpha }\overline{\lambda }}^{\overline{\beta }}A_{\mu }^{\overline{\lambda }}.
    \label{3.24}
\end{align}
On the other hand, \eqref{3.17} yields
\begin{align}
  R_{\beta 0\mu }^{\alpha }-R_{\mu 0\beta }^{\alpha }
  & = A_{\mu ,\beta}^{\alpha }-A_{\beta ,\mu }^{\alpha }-T_{\beta \mu ,0}^{\alpha }
    + A_{\lambda}^{\alpha }T_{\beta \mu }^{\lambda }  \nonumber \\
  & \quad -(T_{\beta \lambda }^{\alpha }A_{\mu }^{\lambda }-T_{\mu \lambda }^{\alpha}A_{\beta }^{\lambda })
    \label{3.25}
\end{align}
By \eqref{3.19}, we obtain
\begin{equation}
  A_{\overline{\beta },\overline{\gamma }}^{\alpha }
  -A_{\overline{\gamma },\overline{\beta }}^{\alpha }
  =A_{\overline{\lambda }}^{\alpha }T_{\overline{\beta }\overline{\gamma }}^{\overline{\lambda }}
  \label{3.26}
\end{equation}

Collecting the above formulas, we have

\begin{lemma}
  \label{lem:5}                 
  The curvature and torsion components of the canonical connection satisfy the following relations:
  \begin{align*}
    R_{\beta \lambda \mu }^{\alpha }
    & = 2\sqrt{-1} ( L_{\overline{\alpha }\lambda }A_{\mu }^{\overline{\beta }}
      -L_{\overline{\alpha }\mu }A_{\lambda}^{\overline{\beta }}) \\
    R_{\beta \overline{\lambda }\overline{\mu }}^{\alpha }
    & = 2\sqrt{-1}(L_{\beta \overline{\lambda }}A_{\overline{\mu }}^{\alpha }
      -L_{\beta \overline{\mu }}A_{\overline{\lambda }}^{\alpha }) \\
    R_{\beta 0\mu }^{\alpha }
    & = A_{\overline{\alpha },\mu }^{\overline{\beta }}
      -A_{\mu ,\overline{\alpha }}^{\overline{\beta }}
      + T_{\overline{\alpha }\overline{\lambda }}^{\overline{\beta }}A_{\mu }^{\overline{\lambda }} \\
    R_{\beta 0\overline{\mu }}^{\alpha }
    & = A_{\overline{\mu },\beta }^{\alpha}-A_{\beta ,\overline{\mu }}^{\alpha }
      -T_{\beta \lambda }^{\alpha }A_{\overline{\mu }}^{\lambda } \\
    R_{\beta \lambda \overline{\mu }}^{\alpha }-R_{\lambda \beta \overline{\mu }}^{\alpha }
    & = T_{\beta \lambda ,\overline{\mu }}^{\alpha }
      + 2\sqrt{-1}(A_{\beta}^{\alpha }L_{\lambda \overline{\mu }}
      -A_{\lambda }^{\alpha }L_{\beta \overline{\mu }}) \\
    R_{\beta 0\mu }^{\alpha }-R_{\mu 0\beta }^{\alpha }
    & = A_{\mu ,\beta}^{\alpha }-A_{\beta ,\mu }^{\alpha }-T_{\beta \mu ,0}^{\alpha }
      + A_{\lambda}^{\alpha }T_{\beta \mu }^{\lambda } \\
    & \quad -(T_{\beta \lambda }^{\alpha }A_{\mu }^{\lambda }-T_{\mu \lambda }^{\alpha}A_{\beta }^{\lambda }).
  \end{align*}
  and
  \begin{equation*}
    A_{\overline{\beta },\overline{\gamma }}^{\alpha }-A_{\overline{\gamma },\overline{\beta }}^{\alpha }
    =A_{\overline{\lambda }}^{\alpha }T_{\overline{\beta }\overline{\gamma }}^{\overline{\lambda }}.
  \end{equation*}
  In particular, if \(\sigma =0\) and \(\Theta^{(2,0)}=0\), then
  \begin{align*}
    R_{\beta 0\mu }^{\alpha }
    & = -A_{\mu ,\overline{\alpha }}^{\overline{\beta }}, \quad
      R_{\beta 0\overline{\mu }}^{\alpha }
      =A_{\overline{\mu },\beta}^{\alpha },  \\ R_{\beta \lambda \overline{\mu }}^{\alpha }
    & = R_{\lambda \beta \overline{\mu }}^{\alpha }, \quad
      A_{\overline{\beta },\overline{\gamma }}^{\alpha }
      =A_{\overline{\gamma },\overline{\beta }}^{\alpha }.
  \end{align*}
\end{lemma}

We may introduce the following curvature notions.
First, the sectional curvature of \(2\)-planes \(P\subset T(M)\) is defined by
\begin{equation}
  K(P)=\langle R(\eta_{1},\eta_{2})\eta_{2},\eta_{1}\rangle  \label{3.26-1}
\end{equation}
where \(\{\eta_{1},\eta_{2}\}\) is any orthonormal basis for \(P\).
We say that a \(2\)-plane \(P\) is horizontal if \(P\subset \) \(H(M)\).
Then the sectional curvature of horizontal \(2\)-planes is called the horizontal sectional curvature and denoted by \(K^{H}(\cdot )\).
In particular,
the sectional curvature of \(J\)-invariant horizontal \(2\)-planes is referred to as the pseudo-Hermitian sectional curvature,
and denoted by \(K_{\theta }(\cdot )\).
Clearly the pseudo-Hermitian sectional curvature is an analogue of the holomorphic sectional curvature of Hermitian manifolds.

Similarly, the (total) Ricci curvature is defined by
\begin{align}
  Ric(X,Y)
  & = trace\{Z\longmapsto R(Z,Y)X\}  \nonumber \\
  & = \sum_{\alpha }\biggl\{\langle R(\eta_{\alpha },Y)X,\eta_{\overline{\alpha }}\rangle
    + \langle R(\eta_{\overline{\alpha }},Y)X,\eta_{\alpha}\rangle
    + \langle R(\eta_{0},Y)X,\eta_{0}\rangle \biggr\}  \nonumber \\
  & = \sum_{\alpha }\biggl\{\langle R(\eta_{\alpha },Y)X,\eta_{\overline{\alpha }}\rangle
    + \langle R(\eta_{\overline{\alpha }},Y)X,\eta_{\alpha}\rangle \biggr\}  \label{3.27}
\end{align}
for any \(X,Y\in T(M)\).
It turns out that \(Ric(X,Y)\) is not symmetric in general. Set
\begin{equation*}
  R_{AB}=Ric(\eta_{A},\eta_{B}).
\end{equation*}
From \eqref{3.9-3-1} and \eqref{3.27}, we have
\begin{align*}
  R_{\lambda \overline{\mu}}
  & = \sum_{\alpha}\biggl\{
    \langle R(\eta_{\alpha},\eta_{\overline{\mu}})\eta_{\lambda},\eta_{\overline{\alpha}}\rangle
    + \langle R(\eta_{\overline{\alpha}},\eta_{\overline{\mu}})\eta_{\lambda},\eta_{\alpha}\rangle \biggr\} \\
  & = \sum_{\alpha}\langle R(\eta_{\alpha},\eta_{\overline{\mu}})\eta_{\lambda},\eta_{\overline{\alpha}}\rangle
    =R_{\lambda \alpha \overline{\mu}}^{\alpha}
\end{align*}
and
\begin{align*}
  R_{\overline{\mu }\lambda }
  & = \sum_{\alpha }\biggl\{
    \langle R(\eta_{\alpha },\eta_{\lambda })\eta_{\overline{\mu }},\eta_{\overline{\alpha }}\rangle
    + \langle R(\eta_{\overline{\alpha }},
    \eta_{\lambda })\eta_{\overline{\mu }},\eta_{\alpha }\rangle \biggr\} \\
  & = \sum_{\alpha }\langle R(\eta_{\overline{\alpha }},
    \eta_{\lambda })\eta_{\overline{\mu }},\eta_{\alpha }\rangle
    =R_{\overline{\mu }\overline{\alpha }\lambda }^{\overline{\alpha }}
\end{align*}
For any
\(X=a^{\lambda }\eta_{\lambda }+b^{\overline{\lambda }}\eta_{\overline{\lambda }}\),
\(Y=c^{\mu }\eta_{\mu }+d^{\overline{\mu }}\eta_{\overline{\mu }}\in H(M)\otimes C\),
we define a \(2\)-tensor
\begin{align*}
  Ric_{b}(X,Y)
  =R_{\lambda \overline{\mu }}a^{\lambda }d^{\overline{\mu }}
  + b^{\overline{\lambda }}c^{\mu }R_{\overline{\lambda }\mu }
\end{align*}
whose components are
\begin{align*}
  Ric_{b}(\eta_{\lambda },\eta_{\overline{\mu }})
  & = R_{\lambda \overline{\mu }}, \quad
    Ric_{b}(\eta_{\overline{\mu }},\eta_{\lambda })
    =R_{\overline{\mu }\lambda } \\ Ric_{b}(\eta_{\lambda },\eta_{\mu })
  & = Ric_{b}(\eta_{\overline{\lambda }},\eta_{\overline{\mu }})
    =0.
\end{align*}
Clearly \(Ric_{b}(X,Y)\in \mathbb{R} \) if \(X\),
\(Y\in H(M)\).
The 2-tensor \(Ric_{b}\) will be referred to as the pseudo-Hermitian Ricci tensor,
which may not be symmetric too.
However we have the following

\begin{lemma}
  \label{lem:6}                 
  If \(\sigma =0\) and \(\Theta^{(2,0)}=0\),
  then \(Ric_{b}\) is a symmetric fiberwise \(2\)-tensor on \(H(M)\).
\end{lemma}

\begin{proof}
  Using Lemmas \ref{lem:4} and \ref{lem:5}, we deduce that
  \begin{align*}
    R_{\overline{\mu }\lambda }
    & = R_{\overline{\mu }\overline{\alpha }\lambda}^{\overline{\alpha }}
      =R_{\overline{\alpha }\overline{\mu }\lambda }^{\overline{\alpha }} \\
    & = \sum_{\alpha }R_{\alpha \overline{\alpha }\overline{\mu }\lambda}
      =\sum_{\alpha }R_{\overline{\alpha }\alpha \lambda \overline{\mu }} \\
    & = R_{\alpha \lambda \overline{\mu }}^{\alpha }
      =R_{\lambda \alpha \overline{\mu }}^{\alpha } \\
    & = R_{\lambda \overline{\mu }}.
  \end{align*}
  This implies immediately that \(Ric_{b}(X,Y)=Ric_{b}(Y,X)\) for any \(X,Y\in H(M)\).
\end{proof}

The (total) scalar curvature is defined by
\begin{equation*}
  \rho_{M}
  =\sum_{\alpha }\left( R_{\alpha \overline{\alpha }}
    + R_{\overline{\alpha }\alpha }\right) .
\end{equation*}
Using \eqref{3.9.3-4}, we find
\begin{equation*}
  \sum_{\alpha }R_{\alpha \overline{\alpha }}
  =\sum_{\alpha ,\beta }R_{\overline{\beta }\alpha \beta \overline{\alpha }}
  =\sum_{\alpha ,\beta}R_{\alpha \overline{\beta }\overline{\alpha }\beta }
  =\sum_{\beta }R_{\overline{\beta }\beta }.
\end{equation*}
Set \(\rho =\sum_{\alpha }R_{\alpha \overline{\alpha }}\).
This function \(\rho\) is called the pseudo-Hermitian scalar curvature.

\section{Pseudo-K\"{a}hler manifolds} \label{sec:pseudo-kahl-mfd}

Recall that the torsion \(T_{\nabla }\) of a linear connection \(\nabla \) in \((M,\theta ,J)\) is said to be \emph{pure} (\cite{dragomir2007diff}) if
\begin{align}
  T_{\nabla }(Z,W)
  & = 0,  \label{4.1} \\ T_{\nabla }(Z,\overline{W})
  & = 2iL_{\theta }(Z,\overline{W})\xi , \label{4.2} \\ J\circ \tau
  + \tau \circ J
  & = 0  \label{4.3}
\end{align}
for any \(Z,W\in H^{1,0}(M)\),
where \(\tau \) is defined as in \eqref{eqn8-6}.
In view of Lemma \ref{lem:1},
it is easy to see that if a connection \(\nabla \) in \((M,\theta ,J)\) preserves the decomposition \eqref{eqn8},
then its torsion is pure if and only if \(\Theta^{(2,0)}=0\),
\(\Theta^{(1,1)}=0\) and $ \sigma =0$.

Using the frame field \(\{\eta_{A}\}\), we have
\begin{equation*}
  \Omega_{h,\theta }(\eta_{\alpha },\eta_{\overline{\beta }})
  =h(J\eta_{\alpha },\eta_{\overline{\beta }})
  =\sqrt{-1}\delta_{\alpha \overline{\beta }}.
\end{equation*}
and thus
\begin{align}
  \Omega_{h,\theta }
  =2\sqrt{-1}\delta_{\alpha \overline{\beta }}\theta^{\alpha }\wedge \theta^{\overline{\beta }}
  =2\sqrt{-1}\sum_{\alpha }\theta^{\alpha }\wedge \theta^{\overline{\alpha }}.  \label{4.4}
\end{align}

\begin{theorem}
  \label{thm:6}                 
  Let \((M^{2m+1},\theta ,J,h)\) be a pseudo-Hermitian manifold
  and let \(\nabla \) be its canonical pseudo-Hermitian connection. Then

  \begin{enumerate}[(i)]
  \item \(i_{\xi }d\Omega_{h,\theta }=0\)
    if and only if \(\sigma =0\) and \(\tau \) is symmetric with respect to \(g_{h,\theta }\).

  \item \(M\) is of horizontally K\"{a}hler type if and only if \(\Theta^{(2,0)}=0 \).

  \item \(M\) is of pseudo-K\"{a}hler type if and only if \(\Theta^{(2,0)}=0\),
    \( \sigma =0\) and \(\tau \) is symmetric with respect to \(g_{h,\theta }\).
  \end{enumerate}
\end{theorem}

\begin{proof}
  Taking exterior derivative of \eqref{4.4} and using \eqref{3.6} yield
  \begin{align}
    \frac{1}{2\sqrt{-1}}d\Omega_{h,\theta }
    & = \sum_{\alpha }(d\theta^{\alpha}\wedge \theta^{\overline{\alpha }}
      -\theta^{\alpha }\wedge d\theta^{\overline{\alpha }})  \nonumber \\
    & = \sum_{\alpha }(-\theta_{\beta }^{\alpha }\wedge \theta^{\beta }\wedge
      \theta^{\overline{\alpha }}
      + \theta^{\alpha }\wedge \theta_{\overline{\beta }}^{\overline{\alpha }}\wedge
      \theta^{\overline{\beta }})
      + \frac{1}{2}\sum_{\alpha }(T^{\alpha }\wedge \theta^{\overline{\alpha }}
      -\theta^{\alpha }\wedge T^{\overline{\alpha }})  \nonumber \\
    & = \sum_{\alpha }\left( \theta \wedge A_{\beta }^{\alpha }\theta^{\beta}
      + \theta \wedge A_{\overline{\beta }}^{\alpha }\theta^{\overline{\beta }}
      + \frac{1}{2}T_{\lambda \mu }^{\alpha }\theta^{\lambda }\wedge
      \theta^{\mu}\right) \wedge \theta^{\overline{\alpha }}  \nonumber \\
    & \quad -\sum_{\alpha }\left( \theta \wedge
      A_{\overline{\beta }}^{\overline{\alpha }}\theta^{\overline{\beta }}
      + \theta \wedge A_{\beta }^{\overline{\alpha }}\theta^{\beta }
      + T_{\overline{\lambda }\overline{\mu }}^{\overline{\alpha }}\theta^{\overline{\lambda }}\wedge
      \theta^{\overline{\mu }}\right) \wedge \theta^{\alpha }.  \label{4.2-1}
  \end{align}
  Using \eqref{4.2-1} and \eqref{3.9-2-1}, we deduce
  \begin{align*}
    i_{\xi }d\Omega_{h,\theta }
    & = \sum_{\alpha }(A_{\beta }^{\alpha }
      + A_{\overline{\alpha }}^{\overline{\beta }})\theta^{\beta }\wedge \theta^{\overline{\alpha }}
      + \sum_{\alpha }\left( A_{\overline{\beta }}^{\alpha}\theta^{\overline{\beta }}\wedge
      \theta^{\overline{\alpha }}-A_{\beta }^{\overline{\alpha }}\theta^{\beta }\wedge
      \theta^{\alpha }\right) \\
    & = 2\sum_{\alpha }A_{\beta }^{\alpha }\theta^{\beta }\wedge \theta^{\overline{\alpha }}
      + \sum_{\alpha }\left( A_{\overline{\beta }}^{\alpha}\theta^{\overline{\beta }}\wedge
      \theta^{\overline{\alpha }}-A_{\beta }^{\overline{\alpha }}\theta^{\beta }\wedge
      \theta^{\alpha }\right) .
  \end{align*}
  Thus \(i_{\xi }\Omega_{h,\theta }=0\) if and only if \(A_{\beta }^{\alpha }=0\) ,
  \(A_{\alpha \beta }=A_{\beta \alpha }\).
  The latter is just
  \begin{equation*}
    \langle \tau (\eta_{\beta }),\eta_{\alpha }\rangle
    =\langle \eta_{\beta},\tau (\eta_{\alpha })\rangle .
  \end{equation*}
  This proves (i).

  From \eqref{4.2-1},
  we also find that \(d_{H}\Omega_{h,\theta }=0\) if and only if
  \begin{align*}
    0
    =\sum_{\alpha }T_{\lambda \mu }^{\alpha }\theta^{\lambda }\wedge
    \theta^{\mu }\wedge \theta^{\overline{\alpha }}
    =\sum_{\alpha }T_{\overline{\lambda }\overline{\mu }}^{\overline{\alpha }}
    \theta^{\overline{\lambda }}\wedge \theta^{\overline{\mu }}\wedge \theta^{\alpha },
  \end{align*}
  that is, \(\Theta^{(2,0)}=0\).

  Clearly (iii) follows immediately from (i) and (ii). Hence we complete the proof of Theorem \ref{thm:6}.
\end{proof}

\begin{remark}
  \label{rmk:2}                 
  \begin{enumerate}[(a)]
  \item Using \eqref{3.9-1} and \eqref{3.9-2-1}, we have
    \begin{align*}
      tr_{h,\theta }(\tau )
      & = \sum_{\alpha }[\langle \tau (\eta_{\alpha }),\eta_{\overline{\alpha }}\rangle
        + \langle \tau (\eta_{\overline{\alpha }}),\eta_{\alpha }\rangle ] \\
      & = A_{\alpha }^{\alpha }
        + A_{\overline{\alpha }}^{\overline{\alpha }}
        =2A_{\alpha }^{\alpha },
    \end{align*}
    which implies that if \(\sigma =0\),
    then \(tr_{h,\theta }\left( \tau \right) =0 \).

  \item By \eqref{eqn6-1} and Cartan's magic formula,
    we see that \(i_{\xi}d\Omega_{h,\theta }=0\) if and only if \(\mathcal{L}_{\xi }\Omega_{h,\theta}=0\);

  \item From Theorem \ref{thm:6} (iii),
    we know that \(T_{\nabla }\) of a pseudo-K\"{a}hler manifold is pure.
    In particular,
    if \(h=L_{\theta }\),
    that is \( \Omega_{h,\theta }=d\theta \) (see Example \ref{ex:1}),
    then the canonical pseudo-Hermitian connection is just the Tanaka-Webster connection,
    whose torsion has been known to satisfy the pure conditions,
    the symmetric property about \(\tau\) and \(tr_{g_{\theta }}\left( \tau \right) =0\)
    (cf., e.g., Theorem 1.3 and Lemma 1.4 in \cite{dragomir2007diff}).
  \item A consequence of Lemma \ref{lem:6} and Theorem \ref{thm:6} (iii) is that \(Ric_{b}\) is symmetric when \(M\) is of pseudo-K\"{a}hler type.
    In this case,
    if \(Ric_{b}=\lambda h\) for some constant \(\lambda \),
    then \(M\) is referred to be \emph{pseudo-Einstein}.
  \end{enumerate}
\end{remark}

By an infinitesimal CR automorphism we mean a real vector field
whose (local) \(1\)-parameter group of (local) transformations of \(M\)
consists of (local) CR automorphisms
(cf. page 56 in \cite{dragomir2007diff}).
From \eqref{eqn11} and \eqref{eqn15},
we get immediately

\begin{lemma}
  \label{lem:7}                 
  Let \((M^{2m+1},\theta ,J,h)\) be a pseudo-Hermitian manifold with \( \sigma =0\).
  Then \(\tau =0\) if and only if \(\xi \) is an infinitesimal CR automorphism,
  that is,
  the \((1,1)\)-tensor field \(J\) is constant along the Reeb leaves.
\end{lemma}

The above lemma generalizes a known result of Webster (\cite{webster1978})
for a non-degenerate CR manifold \((M,\theta ,J,L_{\theta })\).

\begin{proposition}
  \label{prp:1}                 
  Let \((M^{2m+1},\theta ,J,h)\) be a pseudo-Hermitian manifold with \( i_{\xi }d\Omega_{h,\theta }=0\).
  Then the following conditions are equivalent :

  \begin{enumerate}[(i)]
  \item \(\tau =0\).

  \item \(\xi \) is an infinitesimal CR automorphism.

  \item \((M^{2m+1},\mathcal{F}_{\xi },h)\) is a Riemannian foliation.
  \end{enumerate}
\end{proposition}

\begin{proof}
  By \eqref{eqn6-1},
  the condition \(i_{\xi }d\Omega_{h,\theta }=0\) is equivalent to \(\mathcal{L}_{\xi }\Omega_{h,\theta }=0\).
  From Theorem \ref{thm:6} (i),
  we have \(\sigma =0\).
  Using Lemma \ref{lem:7},
  we see that \(\tau =0\) if and only if \(\mathcal{L}_{\xi }J=0\).
  This shows the equivalence of (i) and (ii).
  Next, by definition, we may derive that
  \begin{align*}
    (\mathcal{L}_{\xi }h)(JX,Y)
    & = (\mathcal{L}_{\xi }\Omega_{h,\theta})(X,,Y)-h((\mathcal{L}_{\xi }J)X,Y) \\
    & = -h((\mathcal{L}_{\xi }J)X,Y)
  \end{align*}
  for any \(X,Y\in \Gamma (TM)\).
  It follows that \(\mathcal{L}_{\xi }h=0\) if and only if \(\mathcal{L}_{\xi }J=0\).
  This implies the equivalence of (ii) and (iii).
\end{proof}

\begin{corollary}
  \label{cor:4}                 
  Let \((M^{2m+1},\theta ,J,h)\) be a pseudo-Hermitian manifold of pseudo-K\"{a}hler type.
  Then the same results of Proposition \ref{prp:1} hold.
\end{corollary}

\begin{remark}
  \label{rmk:3}                 
  Suppose \((M,\theta ,J,L_{\theta })\) is a strictly pseudoconvex CR manifold
  endowed with a positive definite pseudo-Hermitian \(1\)-form \(\theta \).
  If \(\tau =0\),
  then \(M\) is the so-called Sasakian manifold.
\end{remark}

\begin{theorem}
  \label{thm:7}                 
  Let \((M^{2m+1},\theta ,J,h)\) be a closed pseudo-Hermitian manifold.
  If \(g_{h,\theta }\) is of pseudo-K\"{a}hler type,
  then the basic cohomology class of \(\Omega_{h,\theta }^{k}\) in \(\mathcal{H}_{b}^{2k}(M)\) is nonzero for \(0\leq k\leq m\).
  In particular,
  \(\mathcal{H}_{b}^{2k}(M)\not=0\) for such \(k\).
\end{theorem}

\begin{proof}
  Since \(\mathcal{H}_{b}^{0}(M)= \mathbb{R} \) (by definition),
  we only need to consider \(1\leq k\leq m\).
  Suppose \( [\Omega_{h,\theta }^{k}]_{b}=0\),
  that is,
  there exists a basic smooth \( (2k-1)\)-form \(\Phi \in \mathcal{A}_{b}^{2k-1}(M)\)
  such that \(\Omega_{h,\theta }^{k}=d\Phi \).
  From the assumption \(d\Omega_{h,\theta }=0\),
  we deduce that
  \begin{equation}
    \Omega_{h,\theta }^{m}
    =d\Phi \wedge \Omega_{h,\theta }^{m-k}
    =d\left( \Phi \wedge \Omega_{h,\theta }^{m-k}\right) .  \label{4.5}
  \end{equation}
  Using \(\theta \wedge \Omega_{h,\theta }^{m}=m!2^{m}dv_{g_{h,\theta }}\) and \eqref{4.5},
  we obtain
  \begin{align}
    m!2^{m}vol(M)
    & = \int_{M}\theta \wedge \Omega_{h,\theta }^{m}  \nonumber \\
    & = \int_{M}\theta \wedge d\left( \Phi \wedge \Omega_{h,\theta }^{m-k}\right) \nonumber \\
    & = \int_{M}\theta \wedge d(\Phi \wedge \Omega_{h,\theta }^{m-k})  \nonumber \\
    & = -\int_{M}d\theta \wedge \Phi \wedge \Omega_{h,\theta }^{m-k}. \label{4.6}
  \end{align}
  Since \(d\theta \wedge \Phi \wedge \Omega_{h,\theta }^{m-k}\in \mathcal{A}_{b}^{2m+1}(M)=\{0\}\),
  \eqref{4.6} is a contradiction.
  This proves the theorem.
\end{proof}

Suppose that \((M^{2m+1},\theta ,J)\) is a non-degenerate CR manifold with the property
that \(\xi \) is an infinitesimal CR automorphism.
Since \( J_{b}\) satisfies the integrability condition \eqref{eqn0-1}
and it is also constant along the Reeb leaves,
we see that \((M,H^{1,0}(M);\mathcal{F})\) is a transversely holomorphic structure
in the sense of \cite{goez-mont1980holo}.
More precisely,
there is an open covering \(\{U_{a}\}_{a\in \Lambda }\) of \(M\) and local (foliated) charts
\(\varphi_{a}:U_{a}\rightarrow V_{a}\subset \mathbb{R} \times \mathbb{C}^{m}\)
such that the leaves of \(F|_{U_{a}}\) are given in local coordinates
\( (x^{(a)},z_{1}^{(a)}, \cdots, z_{m}^{(a)})\in V_{a}\)
by \(z_{\alpha }^{(a)}=const.\) (\(\alpha =1, \cdots, m\,\)).
Furthermore,
\(
(\varphi_{a}\circ \varphi_{b}^{-1})(x,z_{\alpha}^{(b)})
=(f_{ab}(x^{(b)},z_{1}^{(b)},z_{1}^{(b)}...,z_{m}^{(b)})),
g_{ab}(z_{1}^{(b)}...,z_{m}^{(b)}))
\),
where \(g_{ab}\) is a holomorphic function for any \(a,b\in \Lambda \).
Let \( \pi_{\varphi_{a}}:U_{a}\rightarrow \mathbb{C}^{m}\)
be the local submersion defined to be the composition of \(\varphi_{a}\) with the projection
\(\pi_2:\) \( \mathbb{R} \times \mathbb{C}^m\rightarrow \mathbb{C}^m\) on the second factor.
Clearly \(J_{0}=(\pi_{a})_{\ast }J_{b}\) on each \( U_{a}\),
where \(J_{0}\) is the natural complex structure of \( \mathbb{C}^{m}\).
Let \(H^{\ast }(M)\) denote the dual vector bundle of \(H(M)\).
According to the decomposition \(H(M)_{\mathbb{C}}=H^{1,0}(M)\oplus H^{0,1}(M)\),
we have a corresponding decomposition of
\( H^{\ast }(M)_{\mathbb{C}}=H_{1,0}^{\ast }(M)\oplus H_{0,1}^{\ast }(M)\),
and this induces a splitting of the exterior differential algebra over
\(H^{\ast }(M)_{\mathbb{C}}\):
\begin{align*}
  \Lambda^{r}(H^{\ast }(M)_{\mathbb{C}})
  =\bigoplus_{p+q=r}\Lambda^{p,q}(H^{\ast }(M)_{\mathbb{C}}).
\end{align*}
A basic form \(\omega \in A_{b}^{r}(M)\otimes \mathbb{C} \)
can be regarded in a natural way as a global section of
\(\Lambda^{r}(H^{\ast }(M)_{\mathbb{C}})\).
Consequently,
if the basic form \(\omega \in \Gamma (\Lambda^{p,q}(H^{\ast }(M)_{\mathbb{C}})\),
then \(\omega \) is said to be of type \((p,q)\).
Let \(\mathcal{A}_{b}^{p,q}(M)\) denote the space of basic forms of type \((p,q)\).
Set \( d_{b}=d\mid_{A_{b}^{r}(M)\otimes \mathbb{C}}\).
Since \(d_{b}:A_{b}^{r}(M)\otimes \mathbb{C} \rightarrow A_{b}^{r+1}(M)\otimes \mathbb{C} \),
there is a decomposition
\begin{equation*}
  d_{b}=\partial_{b}\oplus \overline{\partial }_{b}
\end{equation*}
with
\begin{equation*}
  \partial_{b}:\mathcal{A}_{b}^{p,q}(M)\rightarrow \mathcal{A}_{b}^{p+1,q}(M) , \quad
  \overline{\partial }_{b}:\mathcal{A}_{b}^{p,q}(M)\rightarrow \mathcal{A}_{b}^{p,q+1}(M)
\end{equation*}
and
\begin{equation*}
  \partial_{b}^{2}=0, \quad
  \overline{\partial }_{b}^{2}=0
  \quad \mbox{and} \quad
  \partial_{b}\overline{\partial }_{b}+\overline{\partial }_{b}\partial =0.
\end{equation*}
In particular, we have the basic Dolbeault complex
\begin{equation*}
  0
  \rightarrow \mathcal{A}_b^{p,0}(M)
  \overset{\overline{\partial}_b}{\rightarrow}
  \mathcal{A}_b^{p,1}(M)
  \rightarrow \cdots
  \overset{\overline{\partial}_b}{\rightarrow}
  \mathcal{A}_b^{p,n}(M)
  \rightarrow 0
\end{equation*}
and the basic Dolbeault cohomology groups \(\mathcal{H}_{b}^{p,q}(\mathcal{F}) \).

\begin{lemma}
  \label{lem:8}                 
  Let \((M^{2m+1},\theta ,J,h)\) be a pseudo-K\"{a}hler manifold with \( \tau =0\).
  Let \(\omega \in \mathcal{A}_{b}^{1,1}(M)\) be a real basic \(2\)-form of type \((1,1)\).
  Then \(\omega \) is closed if and only if every point \(x\in M\)
  has an open neighborhood \(U\) such that the restriction of \(\omega \) to \(U\)
  equals \(i\partial_{b}\overline{\partial }_{b}u\)
  for some real basic function \(u\) on \(U\).
  In particular,
  \(\Omega_{h,\theta }=\sqrt{-1} \partial_{b}\overline{\partial }_{b}f\)
  for some basic function \(f\) around any point.
  Such a function \(f\) will be called the potential function of \( g_{h,\theta }\).
\end{lemma}

\begin{proof}
  By Proposition \ref{prp:1},
  we know that \(\xi \) is an infinitesimal CR automorphism,
  that is,
  \((M,H^{1,0}(M);\mathcal{F})\) is a transversely holomorphic structure.
  Consequently any point \(x\in M\) has a foliated chart
  \(\varphi :U\rightarrow V\subset \mathbb{R} \times \mathbb{C}^{m}\)
  with the corresponding submersion
  \(\pi_{\varphi }:U\rightarrow \pi_{2}(V)\subset \mathbb{C}^{m}\).
  Without loss of generality,
  one may assume that \(\pi_{2}(V)\) is a polycylinder in \( \mathbb{C}^{m}\).
  Since \(\omega \) is basic,
  there exists a \(2\)-form \(\widetilde{\omega } \) of type \((1,1)\) on \(\pi_{2}(V)\)
  such that \(\omega =\pi_{\varphi }^{\ast}(\widetilde{\omega })\).
  Notice that \(\pi_{\varphi }^{\ast }(d\widetilde{\omega })=d\omega =0\),
  which implies that \(d\widetilde{\omega }=0\,\),
  since \( \pi_{\varphi }\) is a submersion.
  Using the \(\partial \overline{\partial }\) -lemma on \(\pi_{2}(V)\),
  we have a real function \(\widetilde{u}\) on \(\pi_{2}(V)\)
  such that \(\widetilde{\omega }=\sqrt{-1}\partial \overline{\partial}\widetilde{u}\).
  It follows immediately that \(\omega =\sqrt{-1}\partial_{b} \overline{\partial }_{b}u\),
  where \(u=\pi_{\varphi }^{\ast }(\widetilde{u})\) .
  This complete the proof of Lemma \ref{lem:8}.
\end{proof}

\begin{theorem}
  \label{thm:8}                 
  Let \((M^{2m+1},\theta ,J,h)\) be a compact pseudo-K\"{a}hler manifold with \(\tau =0\),
  and let \(\Omega \) and \(\Omega^{\prime }\) be basic real closed \((1,1)\)-forms
  such that \([\Omega ]_{b}=[\Omega^{\prime }]_{b}\).
  Then there exists a smooth basic function \(\phi \) such that
  \(\Omega^{\prime}=\Omega +\sqrt{-1}\partial_{b}\overline{\partial }_{b}\phi \).
\end{theorem}

\begin{proof}
  The proof of Lemma \ref{lem:8} tells us that there is a closed \(2\)-form \(\widetilde{\Omega }\)
  on \(\pi_{2}(V)\) for each foliated chart \(\varphi :U\rightarrow V\) such that
  \(\Omega_{h,\theta }=\pi_{\varphi }^{\ast }(\widetilde{\Omega })\) .
  Combining this with Proposition \ref{prp:1},
  we find that the Reeb foliation of \((M^{2m+1},\theta ,J,h)\)
  is a K\"{a}hler foliation.
  Hence this theorem follows immediately from Proposition 3.5.1 of \cite{el1990operat}.
\end{proof}

In view of Remark \ref{rmk:3},
if \((M^{2m+1},\theta ,J,h)\) is a pseudo-K\"{a}hler manifold with \(\tau =0\),
then \(M\) will be referred to a pseudo-K\"{a}hler manifold of Sasakian type.
Before ending this section,
we would like to further investigate some properties of its pseudo-Hermitian sectional curvatures.

Let \((M,\theta ,J,h)\) be a pseudo-K\"{a}hler manifold with \(\tau =0\) and
let \(\nabla \) be its canonical pseudo-Hermitian connection.
For simplicity,
we still denote by \(\langle \cdot ,\cdot \rangle \) the inner products and its complex linear extension,
induced from \(g_{h,\theta }\),
on various tensor bundles of \(M\).
Recall that the curvature tensor \(R(X,Y,Z,W)\) of \(\nabla \)
introduced in \S \ref{sec:struct-equat-bianchi} is defined as follows:
\begin{equation*}
  R(X,Y,Z,W)=\langle R(Z,W)Y,X\rangle
\end{equation*}
for \(X,Y,Z,W\in T(M)\otimes \mathbb{C} \).
By \eqref{3.9.3-4},
we get
\begin{equation}
  R(X,Y,Z,W)=-R(Y,X,Z,W)=-R(X,Y,W,Z).  \label{4.7}
\end{equation}
From Theorem \ref{thm:6} and Lemma \ref{lem:5},
it follows that if any of the vectors \(X,Y,Z\) and \(W\) is vertical, then
\begin{equation}
  R(X,Y,Z,W)=0.  \label{4.8}
\end{equation}
By restricting \(R\) to \(H(M)\),
we obtain a quadrilinear mapping \(R^{H}\)on \( H(M)\) as follows:
\begin{equation}
  R^{H}(X,Y,Z,W)=R(X,Y,Z,W)  \label{4.9}
\end{equation}
for any \(X,Y,Z,W\in H(M)\).
The tensor \(R^{H}\) will be referred to as the \emph{horizontal curvature tensor},
which also fulfills \eqref{4.7} for vectors in \(H(M)\).
Applying Lemma \ref{lem:5} again,
we find that \(R^{H}\) shares the same properties as the curvature tensor of a K\"{a}hler manifold.
It is well-known that the holomorphic sectional curvatures of a K\"{a}hler manifold
determine the curvature tensor completely.
This result depends only on the algebraic properties of the curvature tensor of the K\"{a}hler manifold
(cf. \S 7 in Chapter IX of \cite{kobayashi1996diff}).
Consequently the pseudo-Hermitian sectional curvatures of \(M\) determine \(R^{H}\) too.
Combining \eqref{4.8} and the previous discussion,
we conclude the following

\begin{theorem}
  \label{thm:9}                 
  Let \((M,\theta ,J,h)\) be a pseudo-K\"{a}hler manifold with \(\tau =0\).
  Then the pseudo-Hermitian sectional curvatures of \(M\) determine the curvature tensor completely.
\end{theorem}

\begin{corollary}
  \label{cor:5}                 
  The pseudo-Hermitian sectional curvatures of a Sasakian manifold
  \( (M,\theta ,J,L_{\theta })\) determine the curvature tensor completely.
\end{corollary}

\begin{remark}
  \label{rmk:4}                 
  The sectional curvatures of \(J\)-invariant 2-planes in \(H(M)\)
  can also be introduced for the Riemannian connection \(\nabla^{\theta}\) of \(g_{\theta }\) on a Sasakian manifold.
  This type of sectional curvatures is called \(\phi \)-sectional curvature in the literature (\(\phi =J\)).
  It is known that the curvature tensor of \(\nabla^{\theta }\) for a Sasakian manifold
  is determined completely by the \(\phi \)-sectional curvatures
  (cf. Chapter 7 in \cite{blair2002riem}).
  This result is equivalent to Corollary \ref{cor:5} above,
  because there is a relation between the curvature tensors of \(\nabla \) and \(\nabla^{\theta }\)
  (please refer to \cite{dragomir2007diff} for the curvature relation between these two connections).
\end{remark}

In terms of the inner product \(g_{h,\theta }\) at each point in \(M\),
we set
\begin{align}
  & R_{0}(X,Y,Z,W)  \nonumber \\
  & = \frac{c}{4}\{\langle \pi_{H}X,\pi_{H}Z\rangle \langle \pi_{H}Y,\pi_{H}W\rangle
    -\langle \pi_{H}X,\pi_{H}W\rangle \langle \pi_{H}Y,\pi_{H}Z\rangle  \nonumber \\
  & \quad
    + \langle \pi_{H}X,J\pi_{H}Z\rangle \langle \pi_{H}Y,J\pi_{H}W\rangle
    -\langle \pi_{H}X,J\pi_{H}W\rangle \langle \pi_{H}Y,J\pi_{H}Z\rangle \nonumber \\
  & \quad
    + 2\langle \pi_{H}X,J\pi_{H}Y\rangle \langle \pi_{H}Z,J\pi_{H}W\rangle \}  \label{4.11}
\end{align}
for \(X,Y,Z,W\in T(M)\otimes \mathbb{C} \).
Clearly, if any of the vectors \(X,Y,Z\) and \(W\) is vertical, then
\begin{equation}
  R_{0}(X,Y,Z,W)=0.  \label{4.12}
\end{equation}
Similarly,
we obtain a quadrilinear mapping \(R_{0}^{H}\) on \(H(M)\) by restricting \(R_{0}\) to \(H(M)\).
It is easy to see from \eqref{4.11} that \( R_{0}^{H}\) shares the same properties as the curvature tensor of a K\"{a}hler manifold and
\begin{equation}
  R_{0}^{H}(e,Je,e,Je)=c  \label{4.13}
\end{equation}
for any unit vector \(e\in H(M)\).
From \eqref{4.8}, \eqref{4.11}, \eqref{4.12}, \eqref{4.13} and Theorem \ref{thm:9}
(see also Proposition 7.3 in \cite{kobayashi1996diff}), we have

\begin{proposition}
  \label{prp:2}                 
  Let \((M,\theta ,J,h)\) be a pseudo-K\"{a}hler manifold with \(\tau =0\) .
  If \(M\) has constant pseudo-Hermitian sectional curvature \(c\),
  then \( R=R_{0} \).
\end{proposition}

By \eqref{4.11} and Proposition \ref{prp:2},
we see that if \(M\) is a pseudo-K\"{a}hler manifold with \(\tau =0\)
and constant pseudo-Hermitian sectional curvature \(c\), then
\begin{equation}
  K^{H}(P)=\frac{c}{4}\left( 1+3\langle \eta_{1},J\eta_{2}\rangle^{2}\right)
  \label{4.14}
\end{equation}
for any \(P=span\{\eta_{1},\eta_{2}\}\subset P\),
where \(\{\eta_{1},\eta_{2}\}\) is an orthonormal basis of \(P\).
Then \eqref{4.14} gives
\begin{equation*}
  Ric_{b}=\lambda h
\end{equation*}
with \(\lambda =\frac{1}{2}(m+1)c\), that is,
\(M\) is a pseudo-Einstein manifold.
We will investigate pseudo-Einstein manifolds defined in Remark \ref{rmk:2}
from a variational point of view in a forthcoming paper.

\section{The Cartan-type theorems} \label{sec:cartan-type-theorems}

It is known from Cartan's theorem (\cite{cartan1951lecon}) that the metric of a Riemannian manifold is,
in some sense,
determined locally by its curvature.
In this section,
we try to generalize this result to pseudo-Hermitian manifolds.

A smooth map
\(f: (M,\theta ,J,h)\rightarrow (N,\widetilde{\theta },\widetilde{J},\widetilde{h})\)
between two pseudo-Hermitian manifolds of the same dimension
is called a local \emph{pseudo-Hermitian isometry}
if \(df_{p}:T_{p}(M)\rightarrow T_{f(p)}(N)\) is a\ pseudo-Hermitian linear isometry at every point \(p\in M\),
which means
\begin{align}
  g_{\widetilde{h},\widetilde{\theta }}(df_{p}(X),df_{p}(Y))
  & = g_{h,\theta}(X,Y), \quad X,Y\in T_{p}M,  \nonumber \\
  df_{p}\circ J_{p}
  & = \widetilde{J}_{f(p)}\circ df_{p}.  \label{5.4-0}
\end{align}
In particular,
this means that \(f\) is a special local Riemannian isometry.
Since \(J\xi =0\),
we see from \eqref{5.4-0} that \(\widetilde{J}(df(\xi ))=0\),
and therefore \(df(\xi )=\pm \xi \).
Furthermore, 
if \(f\) is a diffeomorphism, 
then we say that \(f\) is a \emph{pseudo-Hermitian isometry}.

To establish a Cartan-type theorem for pseudo-Hermitian manifolds,
we first apply some basic results about geodesics of a linear connection (cf. \cite{kobayashi1996diff1}) to obtain the corresponding results for the canonical connection of a pseudo-Hermitian manifold.
Let \((M,\theta ,J,h)\) be a pseudo-Hermitian manifold with its canonical connection \(\nabla \).
A smooth curve \(\gamma :[0,l]\rightarrow M\) is called a \(\nabla \)\emph{-geodesic}
if \( \nabla_{\gamma^{\prime }}\gamma^{\prime }=0\) on \([0,l]\).
Since \(\nabla \) preserves \(g_{h,\theta }\),
the length of the tangent vector \(\gamma^{\prime}(t)\) is constant.
A parametrization which makes \(\gamma \) into a geodesic,
if any,
is determined up to an affine transformation of \(t\).
When the parameter is actually arc length,
that is \(|\gamma^{\prime }(t)|=1\),
we say that the geodesic is \emph{normalized}.
For any point \(p\in M\) and a vector \( u\in T_{p}(M)\),
by the ODE theory,
there exists a unique \(\nabla \)-geodesic \( \gamma_{u}(t)\)
such that \(\gamma_{u}(0)=p\) and \(\gamma_{u}^{\prime }(0)=u\) .
The exponential map \(\exp_{p}^{\nabla }:T_{p}(M)\rightarrow M\) of the connection is defined by
\begin{equation*}
  \exp_{p}^{\nabla }(u)=\gamma_{u}(1)
\end{equation*}
for all \(u\in T_{p}(M)\) such that \(1\) lies in the domain of \(\gamma_{u}\).
Clearly \(\widehat{\gamma }(t)=\gamma_{u}(\lambda t)\) with \(\lambda \) a constant
is a geodesic with \(\widehat{\gamma }(0)=p\) and \(\widehat{\gamma }^{\prime }(0)=\lambda u\).
Thus \(\widehat{\gamma }(t)=\gamma_{\lambda u}(t)\).
Consequently we have
\begin{equation*}
  \exp_{p}^{\nabla }(tu)=\gamma_{u}(t)
\end{equation*}
for \(0\leq t\leq 1\).
From Proposition 8.2 in \cite{kobayashi1996diff1},
we know that there exists a neighborhood \(D_{p}\) of the zero vector in \(T_{p}(M)\)
which is mapped diffeomorphically onto a neighborhood \(U_{p}\) of \(p\) in \(M\)
by the exponential map.
Then,
for all \(q\in U_{p}\),
there exists a unique normalized geodesic
\(\gamma :[0,t(q)]\rightarrow U_{p}\subset M\) with \( \gamma (0)=p\),
\(\gamma (t(q))=q\).
The neighborhood \(U_{p}\) is called a normal neighborhood of \(p\).


Now, let us consider a 1-parameter family of \(\nabla \)-geodesics given
by a map \(\beta (t,s):[a,b]\times (-\varepsilon ,\varepsilon )\rightarrow M\)
such that for each fixed \(s\),
\(\beta (t,s)\) is a \(\nabla \)-geodesic.
Set \( T=d\beta (\frac{\partial }{\partial t})\)
and \(V=d\beta (\frac{\partial }{\partial s})\).
By \([\frac{\partial }{\partial t},\frac{\partial }{\partial s}]=0\),
we have
\begin{equation}
  \nabla_{T}V-\nabla_{V}T=T_{\nabla }(T,V)  \label{5.1}
\end{equation}
Thus
\begin{equation}
  \nabla_{T}\nabla_{T}V=\nabla_{T}\nabla_{V}T+\nabla_{T}(T_{\nabla}(T,V)).  \label{5.2}
\end{equation}
From the definition of curvature tensor and the geodesic equation \(\nabla_{T}T=0\),
we find that \(V\) satisfies
\begin{equation}
  \nabla_{T}\nabla_{T}V
  =R(T,V)T
  + (\nabla T_{\nabla })(T,V;T)
  + T_{\nabla}(T,\nabla_{T}V)  \label{5.3}
\end{equation}
where \((\nabla T_{\nabla })(T,V;T)=(\nabla_{T}T_{\nabla })(T,V)\).
A vector field \(V\) satisfying \eqref{5.3} along a geodesic \(\gamma \)
is called a \emph{Jacobi field }of the connection.
Since \eqref{5.3} is an ODE system of second order,
a Jacobi field \(V\) is determined uniquely by \(V(0)\) and \( V^{\prime }(0)\).
We have already shown that the variation field of a \(1\)-parameter family of
\(\nabla \)-geodesics is a Jacobi field.
Conversely,
if \( V \) is a Jacobi field along a geodesic \(\gamma \),
then \(V\) comes from a variation of geodesics too.
To prove the converse,
let \(c(s)\) be a curve from \(p\) such that \(c(0)=\gamma (0)=p\) and \(c^{\prime }(0)=V(0)\).
We extend \( \gamma^{\prime }(0)\),
\(V^{\prime }(0)\) and \(T_{\nabla }(\gamma^{\prime}(0),V(0))\) respectively to
parallel fields \(\gamma^{\prime }(0)_{s}\),
\( V^{\prime }(0)_{s}\) and \(T_{\nabla }(\gamma^{\prime }(0),V(0))_{s}\) along the curve \(c(s)\).
Define a family of geodesics by
\begin{equation*}
  \beta (t,s)
  = \exp_{c(s)}^{\nabla }\left\{ t\left[ \gamma^{\prime}(0)_{s}
      + sV^{\prime }(0)_{s}
      -sT_{\nabla }(\gamma^{\prime }(0),V(0))_{s}\right] \right\} .
\end{equation*}
Then the variation field of \(\beta (t,s)\) is a Jacobi field given by \( \widetilde{V}(t)=d \beta (\frac{\partial }{\partial s})\mid_{s=0}\).
Clearly \(\widetilde{V}(0)=c^{\prime }(0)=V(0)\).
Using \eqref{5.1}, we obtain
\begin{align*}
  \widetilde{V}^{\prime }(0)
  & = \left\{ \nabla_{\frac{\partial }{\partial t}}d \beta  (\frac{\partial }{\partial s})\mid_{s=0}\right\}
    \mid_{t=0} \\
  & = \left\{ \nabla_{\frac{\partial }{\partial s}}\left[ \gamma^{\prime}(0)_{s}
    + sV^{\prime }(0)_{s}-sT_{\nabla }(\gamma^{\prime }(0),V(0))_{s}\right] \right\} \mid_{s=0} \\
  & \quad + T_{\nabla }(\gamma^{\prime }(0),V(0)) \\
  & = V^{\prime }(0).
\end{align*}
Thus \(\widetilde{V}(t)=V(t)\).
In particular,
we see that a Jacobi field \( V(t) \) with \(V(0)=0\) and \(V^{\prime }(0)=v\ \) is given by
\begin{equation}
  V(t)
  =\frac{\partial }{\partial s}\exp_{p}^{\nabla }(t(\gamma^{\prime}(0)
  + sV^{\prime }(0)))|_{s=0}
  =\left( d\exp_{p}^{\nabla }\right)_{t\gamma^{\prime }(0)}(tv).  \label{5.4}
\end{equation}

Let \((M^{2m+1},\theta ,J,h)\) and
\((\widetilde{M}^{2m+1},H^{1,0}(\widetilde{M}),\widetilde{\theta },\widetilde{h})\)
be two pseudo-Hermitian manifolds of dimension \(2m+1\) and
let \(p\in M\) and \(\widetilde{p}\in \widetilde{M}\).
Choose a pseudo-Hermitian linear map
\(\rho :T_{p}(M)\rightarrow T_{\widetilde{p}}(\widetilde{M})\).
Let \(U\subset M\) be a normal neighborhood of \(p\) such that
\(\exp_{\widetilde{p}}^{\widetilde{\nabla }}\) is defined on
\(\rho \circ (\exp_{p}^{\nabla })^{-1}(U)\).
Define a map \(f:U\rightarrow \widetilde{M}\) by
\begin{equation}
  f(q)
  = \exp_{\widetilde{p}}^{\widetilde{\nabla }}\circ
  \rho \circ (\exp_{p}^{\nabla })^{-1}(q), \quad
  \forall q\in U.  \label{5.5}
\end{equation}
For each \(q\in U\),
there exists a unique normalized \(\nabla \)-geodesic \( \gamma :[0,t]\rightarrow M\) with \(\gamma (0)=p\),
\(\gamma (t)=q\).
Denote by \( P_{t}\) the parallel transport along \(\gamma \)
with respect to \(\nabla \) from \(\gamma (0)\) to \(\gamma (t)\).
Clearly
\(P_t:(T_{\gamma (0)}(M),J_{\gamma (0)},g_{h,\theta ;\gamma (0)})
\rightarrow (T_{\gamma (t)}(M),J_{\gamma (t)},g_{h,\theta ;\gamma (t)})\)
is a pseudo-Hermitian linear isometry.
Define \(\phi_{t}:T_{q}(M)\rightarrow T_{f(q)}(\widetilde{M})\) by
\begin{equation}
  \phi_{t}(w)
  =\widetilde{P}_{t}\circ \rho \circ P_{t}^{-1}(w), \quad w\in T_{q}M,   \label{5.6}
\end{equation}
where \(\widetilde{P}_{t}\) is the parallel transport
along the \(\widetilde{\nabla }\)-geodesic
\(\widetilde{\gamma }:[0,t]\rightarrow \widetilde{M}\)
given by \(\widetilde{\gamma }(0)=\widetilde{p}\),
\(\widetilde{\gamma }^{\prime }(0)=\rho (\gamma^{\prime }(0))\).
Since parallel transports with respect to the canonical connections are pseudo-Hermitian linear isometries, 
we conclude that \(\phi_{t}:T_{q}(M)\rightarrow T_{f(q)}(\widetilde{M})\) is a pseudo-Hermitian linear isometry.

\begin{lemma}
  \label{lem:9}                 
  Let \(M\) and \(\widetilde{M}\) be as above and
  let \(\Phi \) and \(\widetilde{\Phi}\)
  be \(k\)-tensor fields on \(M\) and \(\widetilde{M}\) respectively.
  With the notations above,
  if for each \(q\in U\) and all \(X_{1}, \cdots, X_{k}\in T_{q}M\),
  we have
  \begin{equation*}
    \Phi (X_{1}, \cdots, X_{k})=\widetilde{\Phi }(\phi_{t}(X_{1}), \cdots, \phi_{t}(X_{k})),
  \end{equation*}
  then
  \begin{equation*}
    (\nabla_{\gamma^{\prime }}\Phi )(X_{1}, \cdots, X_{k})
    =(\widetilde{\nabla }_{\widetilde{\gamma }^{\prime }}\widetilde{\Phi })
    (\phi_{t}(X_{1}), \cdots, \phi_{t}(X_{k})).
  \end{equation*}
\end{lemma}

\begin{proof}
  Let \(v_{l}\in T_{p}(M)\) and \(\widetilde{v}_{l}=\rho (v_{l})\).
  Set \( X_{l}(t)=P_{t}(v_{l})\) and \(\widetilde{X}_{l}(t)=P_{t}(\widetilde{v}_{l})\),
  \( l=1,\cdots, k\).
  Then the assumption for \(\Phi \) and \(\widetilde{\Phi }\) is equivalent to
  \begin{equation*}
    \Phi (X_{1}(t), \cdots, X_{k}(t))
    =\widetilde{\Phi }(\widetilde{X}_{1}(t), \cdots, \widetilde{X}_{k}(t)).
  \end{equation*}
  It follows that
  \begin{equation*}
    \frac{d}{dt}[\Phi (X_{1}(t), \cdots, X_{k}(t))]
    =\frac{d}{dt}[\widetilde{\Phi }(\widetilde{X}_{1}(t), \cdots, \widetilde{X}_{k}(t))]
  \end{equation*}
  that is,
  \begin{equation*}
    (\nabla_{\gamma^{\prime }}\Phi )(X_{1}(t), \cdots, X_{k}(t))
    =(\widetilde{\nabla}_{\widetilde{\gamma }^{\prime }}\widetilde{\Phi })
    (\widetilde{X}_{1}(t), \cdots, \widetilde{X}_{k}(t))
  \end{equation*}
  since \(\nabla_{\gamma^{\prime }}X_{l}
  =\widetilde{\nabla }_{\widetilde{\gamma }^{\prime }}\widetilde{X}_{l}=0\),
  \(l=1, \cdots, k\).
  This proves the lemma.
\end{proof}

\begin{theorem}
  \label{thm:10}                
  If for each \(q\in U\) and all \(X,Y,Z,W\in T_{q}M\), we have
  \begin{gather*}
    \langle R(X,Y)Z,W\rangle
    =\langle \widetilde{R}(\phi_{t}(X),\phi_{t}(Y))\phi_{t}(Z),\phi_{t}(W)\rangle, \\
    \langle T_{\nabla }(X,Y),Z\rangle
    =\langle \widetilde{T}_{\widetilde{\nabla }}(\phi_{t}(X),\phi_{t}(Y)),\phi_{t}(Z)\rangle
  \end{gather*}
  then \(f:U\rightarrow f(U)\subset \widetilde{M}\) is a local pseudo-Hermitian isometry
  and \(df_{p}=\rho \).
\end{theorem}

\begin{proof}
  For each point \(q\in U\),
  let \(\gamma :[0,l]\rightarrow M\) be a normalized \( \nabla \)-geodesic
  in the normal neighborhood \(U\) with \(\gamma (0)=p\),
  \(\gamma (l)=q\).
  Let \(w\in T_{q}(M)\) be any vector.
  Since \(\exp_{p}^{\nabla }\) is a diffeomorphism
  from \((\exp_{p}^{\nabla })^{-1}(U)\subset T_{p}(M)\) to \( U\),
  we may find a unique
  \(v\in (\exp_{p}^{\nabla })^{-1}(U)\subset T_{p}(M)\simeq T_{l\gamma^{\prime }(0)}(T_{p}(M))\)
  such that \(\left( d\exp_{p}^{\nabla }\right)_{l\gamma^{\prime }(0)}(lv)=w\).
  Consequently,
  according to \eqref{5.4},
  there exists a unique Jacobi field \(V(t)\) along \(\gamma \)
  that satisfies \(V(0)=0\) and \( V(l)=w\).
  Let \(s_{1}, \cdots, s_{2m+1}=\gamma^{\prime }(0)\)
  be an orthonormal basis of \(T_{p}M\) and let \(\eta_{i}(t)\),
  \(i=1, \cdots, 2m+1\),
  be the parallel transport of \(s_{i}\) along \(\gamma \).
  In terms of the frame field \(\{\eta_{i}(t)\}_{i=1}^{2m+1}\),
  we may write
  \begin{equation}
    V(t)=\sum_{i=1}^{2m+1}y_{i}(t)\eta_{i}(t).  \label{5.7}
  \end{equation}
  Using the Jacobi equation \eqref{5.3}, we get
  \begin{align}
    y_j^{\prime \prime}
    & = \sum_{i=1}^{2m+1}\{\langle T_{\nabla}(\eta_{2m+1},\eta_i),\eta_j\rangle y_i^{\prime}
      +\langle R(\eta_{2m+1},\eta_i)\eta_{2m+1},\eta_j\rangle y_i \nonumber \\
    & \quad
      + \langle (\nabla T_{\nabla})(\eta_{2m+1},\eta_i;\eta_{2m+1}),\eta_j\rangle y_i\},
      \label{5.8}
  \end{align}
  for \(j=1, \cdots, 2m+1\).

  Let \(\widetilde{\gamma }:[0,l]\rightarrow \widetilde{M}\)
  be the normalized \( \nabla \)-geodesic with \(\widetilde{\gamma}(0)=\widetilde{p}\),
  \(\widetilde{\gamma }^{\prime }(0)=\rho (\gamma^{\prime }(0))\).
  Set
  \begin{equation*}
    \widetilde{V}(t)=\phi_{t}(V(t)), \quad
    \widetilde{\eta }_{i}(t)=\phi_{t}(\eta_{i}(t)) \quad
    (i=1, \cdots, 2m+1),
  \end{equation*}
  for \(t\in [0,l]\).
  By the linearity of \(\phi_{t}\),
  we obtain from \eqref{5.7} that
  \begin{equation}
    \widetilde{V}(t)=\sum_{i=1}^{2m+1}y_{i}(t)\widetilde{\eta }_{i}(t).
    \label{5.9}
  \end{equation}
  The hypotheses imply
  \begin{align*}
    \langle R(\eta_{2m+1},\eta_{i})\eta_{2m+1},\eta_{j}\rangle
    & = \langle \widetilde{R}(\widetilde{\eta }_{2m+1},\widetilde{\eta }_{i})\widetilde{\eta}_{2m+1},
      \widetilde{\eta }_{j}\rangle , \\
    \langle T_{\nabla }(\eta_{2m+1},\eta_{i}),\eta_{j}\rangle
    & = \langle \widetilde{T}_{\widetilde{\nabla }}(\widetilde{\eta }_{2m+1},\widetilde{\eta}_{i}),
      \widetilde{\eta }_{j}\rangle ,
  \end{align*}
  and thus Lemma \ref{lem:9} gives
  \begin{equation*}
    \langle (\nabla T_{\nabla })(\eta_{2m+1},\eta_{i};\eta_{2m+1}),\eta_{j}\rangle
    =\langle (\widetilde{\nabla }T_{\widetilde{\nabla }})
    (\widetilde{\eta }_{2m+1},\widetilde{\eta }_{i};\widetilde{\eta }_{2m+1}),
    \widetilde{\eta }_{j}\rangle .
  \end{equation*}
  Consequently we have
  \begin{align*}
    y_{j}^{\prime \prime }
    & = \sum_{i=1}^{2m+1}\{
      \langle
      \widetilde{T}_{\widetilde{\nabla }}(\widetilde{\eta }_{2m+1},\widetilde{\eta }_{i}),
      \widetilde{\eta }_{j}
      \rangle y_{i}^{\prime }
      + \langle
      \widetilde{R}(\widetilde{\eta }_{2m+1},\widetilde{\eta }_{i})\widetilde{\eta }_{2m+1},
      \widetilde{\eta }_{j}
      \rangle y_{i} \\
    & \quad +\langle
      (\widetilde{\nabla }T_{\widetilde{\nabla }})
      (\widetilde{\eta }_{2m+1},\widetilde{\eta }_{i};\widetilde{\eta }_{2m+1}),
      \widetilde{\eta }_{j}
      \rangle y_{i}
      \}
  \end{align*}
  for \(j=1, \cdots, 2m+1\).
  It follows that \(\widetilde{V}(t)\) is a Jacobi field along \(\widetilde{\gamma }\)
  with \(\widetilde{V}(0)=0\).
  Since \(\widetilde{P}_{t}^{-1}(\widetilde{V}(t))=\rho \circ P_{t}^{-1}(V(t))\),
  we obtain that \( \widetilde{V}^{\prime }(0)=\rho (V^{\prime }(0))\).
  On the other hand,
  since \( V(t)\) and \(\widetilde{V}(t)\) are Jacobi fields vanishing at \(t=0\),
  we have from \eqref{5.4} that
  \begin{align*}
    V(t)
    & = \biggl( d\exp_{p}^{\nabla }\biggr)_{t\gamma^{\prime}(0)}(tV^{\prime }(0)), \\
    \widetilde{V}(t)
    & = \biggl( d\exp_{\widetilde{p}}^{\widetilde{\nabla }}\biggr)_{t\widetilde{\gamma }^{\prime }(0)}
      (t\widetilde{V}^{\prime }(0)).
  \end{align*}
  Therefore,
  \begin{align*}
    \widetilde{V}(l)
    & = \biggl( d\exp_{\widetilde{p}}^{\widetilde{\nabla }}\biggr)_{l\widetilde{\gamma }^{\prime }(0)}
      (l\widetilde{V}^{\prime }(0)) \\
    & = \biggl( d\exp_{\widetilde{p}}^{\widetilde{\nabla }}\biggr)_{l\widetilde{\gamma }^{\prime }(0)}
      (\rho (lV^{\prime }(0))) \\
    & = \biggl( d\exp_{\widetilde{p}}^{\widetilde{\nabla }}\biggr)_{l\widetilde{\gamma }^{\prime }(0)}
      \circ \rho \circ \biggl( d\exp_{p}^{\nabla }\biggr)_{l\gamma^{\prime }(0)}^{-1}(V(l))
      =df_{q}(V(l)),
  \end{align*}
  that is,
  \begin{equation}
    \phi_{l}(w)=df_{q}(w), \quad w\in T_{q}(M),  \label{5.10}
  \end{equation}
  where \(\phi_{l}\) is the pseudo-Hermitian linear map given by \eqref{5.6}. 
  From \eqref{5.10}, 
  we know that \(df_{q}:T_{q}(M)\rightarrow T_{f(q)}(N)\) is a pseudo-Hermitian linear isometry.
  This proves that
  \(f:U\rightarrow f(U)\) is a local pseudo-Hermitian isometry and \( df_{p}=\rho \).
\end{proof}

Notice that if \(\exp_{p}\) and \(\exp_{\widetilde{p}}\) are diffeomorphisms,
then, under the conditions of Theorem \ref{thm:10},
\(f\) is a global pseudo-Hermitian isometry from \(M\) to $\widetilde{M}$.

\begin{corollary}
  \label{cor:6}                 
  Suppose both \(M^{2m+1}\) and \(\widetilde{M}^{2m+1}\) are of pseudo-K\"{a}hler type
  and \(\rho :T_{p}(M)\rightarrow T_{\widetilde{p}}(\widetilde{M})\)
  is a pseudo-Hermitian linear map with \(\rho (\xi )=\widetilde{\xi }\).
  If for each \(q\in U\) and all \(X,Y,Z,W\in T_{q}M\),
  we have
  \begin{align}
    \langle R(X,Y)Z,W\rangle
    & = \langle \widetilde{R}(\phi_{t}(X),\phi_{t}(Y))\phi_{t}(Z),\phi_{t}(W)\rangle ,  \nonumber \\
    d\theta (X,Y)
    & = d\widetilde{\theta }(\phi_{t}(X),\phi_{t}(Y)), \label{5.11} \\
    \langle \tau (X),Y\rangle
    & = \langle \widetilde{\tau }(\phi_{t}(X)),\phi_{t}(Y)\rangle ,  \nonumber
  \end{align}
  then \(f:U\rightarrow f(U)\subset \widetilde{M}\) is a local pseudo-Hermitian isometry and \(df_{p}=\rho \).
\end{corollary}

\begin{proof}
  By \eqref{eqn22} and Theorem \ref{thm:6}, we have
  \begin{equation}
    \label{5.12}
    \begin{aligned}
      T_{\nabla }(\cdot ,\cdot )
      & =2\theta \wedge \tau
        + 2d\theta (\cdot ,\cdot )\xi ,  \\
      \widetilde{T}_{\widetilde{\nabla }}(\cdot ,\cdot )
      & =2\widetilde{\theta }\wedge \widetilde{\tau }
        + 2d\widetilde{\theta }(\cdot ,\cdot )\widetilde{\xi}.
    \end{aligned}
  \end{equation}
  Since \(\phi_{t}:T_{q}(M)\rightarrow T_{f(q)}(\widetilde{M})\)
  is a pseudo-Hermitian linear isometry,
  the assumption \(\rho (\xi_{p})=\widetilde{\xi }_{p}\) implies that
  \(\phi_{t}(\xi_{q})=\widetilde{\xi }_{f(q)}\).
  Thus
  \begin{equation*}
    \theta (X)=\widetilde{\theta }(\phi_{t}(X)).
  \end{equation*}
  Combining this with the second and third conditions in \eqref{5.11},
  it is easy to see from \eqref{5.12} that
  \(T_{\nabla }\) and \(\widetilde{T}_{\widetilde{\nabla }}\)
  satisfy the conditions in Theorem \ref{thm:10}.
  Hence this corollary follows immediately from Theorem \ref{thm:10}.
\end{proof}

From Corollary \ref{cor:6},
we know that the pseudo-Hermitian structure of a pseudo-K\"{a}hler manifold
is determined locally by its curvature,
pseudo-Hermitian torsion \(1\)-form and Levi form.
In particular, we have

\begin{theorem}
  \label{thm:11}                
  Let \((M^{2m+1},\theta ,J,L_{\theta })\) and \((\widetilde{M}^{2m+1},\widetilde{\theta },\widetilde{J},L_{\widetilde{\theta }})\) be two strictly pseudoconvex CR manifolds.
  Let \(\rho :(T_{p}(M),g_{\theta _{p}})\rightarrow (T_{\widetilde{p}}(\widetilde{M}),\widetilde{g}_{\widetilde{\theta }_{\widetilde{p}}})\) be a pseudo-Hermitian linear isometry with \(\rho (\xi _{p})=\widetilde{\xi }_{\widetilde{p}}\) for some points \(p\in M\) and \(\widetilde{p}\in \widetilde{M}\).
  Let \(U,f,\phi _{t}\) be as above.
  If for each \(q\in U\),
  and all \(X,Y,Z,W\in T_{q}M\),
  we have
  \begin{align}
    \label{5.13}
    \begin{aligned}
      \langle R(X,Y)Z,W\rangle
      & =\langle \widetilde{R}(\phi_{t}(X),\phi_{t}(Y))\phi_{t}(Z),\phi_{t}(W)\rangle , \\
      \langle \tau (X),Y\rangle
      & =\langle \widetilde{\tau }(\phi_{t}(X)),\phi_{t}(Y)\rangle ,
    \end{aligned}
  \end{align}
  then \(f:U\rightarrow f(U)\subset \widetilde{M}\) is a local pseudo-Hermitian isometry
  and \(df_{p}=\rho \).
\end{theorem}

\begin{proof}
  For pseudo-Hermitian manifolds \((M,\theta ,J,L_{\theta })\) and \((\widetilde{M},\widetilde{\theta },\widetilde{J},L_{\widetilde{\theta }})\),
  their canonical connections \( \nabla \) and \(\widetilde{\nabla }\) are the Tanaka-Webster connections.
  Since the Tanaka-Webster connections preserve the Levi forms (cf.
  \cite{dragomir2007diff}),
  the second condition in \eqref{5.11} holds automatically.
  This proves Theorem \ref{thm:11}.
\end{proof}

From Theorem \ref{thm:11}, we conclude, in particular,
that the pseudo-Hermitian structure of a Sasakian manifold is determined locally by its curvature.
As applications, we have the following two corollaries.
In particular,
Corollary \ref{cor:8} below shows that a Sasakian manifold of constant pseudo-Hermitian sectional curvature
is rich in pseudo-Hermitian isometries.

\begin{corollary}
  \label{cor:7}                 
  Let \((M^{2m+1},\theta ,J,L_{\theta })\) and \((\widetilde{M}^{2m+1},\widetilde{\theta },\widetilde{J},L_{\widetilde{\theta }})\) be Sasakian manifolds
  with the same constant pseudo-Hermitian sectional curvature.
  Let \(p\in M\) and \(\widetilde{p}\in \widetilde{M}\) and
  let \(\{\xi ,e_{1}, \cdots, e_{m}\), \(Je_{1}, \cdots, Je_{m}\} \in T_{p}(M)\) and
  \(\{\widetilde{\xi },\widetilde{e}_{1}, \cdots, \widetilde{e}_{m} \),
  \(\widetilde{J}\widetilde{e}_{1}, \cdots, \widetilde{J}\widetilde{e}_{m}\}
  \in T_{\widetilde{p}}(\widetilde{M})\)
  be arbitrary orthonormal bases.
  Then there exist a neighborhood \(U\subset M\) of \(p\),
  and a neighborhood \( \widetilde{U}\subset \widetilde{M}\) of \(\widetilde{p}\),
  and a pseudo-Hermitian isometry \(f:U\rightarrow \widetilde{U}\)
  such that \( df_{p}(\xi )=\widetilde{\xi }\),
  \(df_{p}(e_{i})=\widetilde{e}_{i}\) and \( df_{p}(Je_{i})=J\widetilde{e}_{i}\) (\(1\leq i\leq m\)).
\end{corollary}

\begin{proof}
  By Proposition \ref{prp:2},
  we know that the curvature tensors of \(M\) and \(\widetilde{M}\)
  have the same expression \eqref{4.11}.
  Let \(\rho :T_{p}(M)\rightarrow T_{\widetilde{p}}(\widetilde{M})\)
  be the pseudo-Hermitian linear isometry such that \(df_{p}(\xi )=\widetilde{\xi }\),
  \(df_{p}(e_{i})=\widetilde{e}_{i}\) and
  \( df_{p}(Je_{i})=J\widetilde{e}_{i}\) (\(1\leq i\leq m\)).
  Since the canonical connections \(\nabla \) and \(\widetilde{\nabla }\)
  preserve the pseudo-Hermitian structures,
  each term on the right hand side of \eqref{4.11}
  is invariant under the map \(\phi_{t}\).
  Then this corollary follows immediately from Theorem \ref{thm:11}.
\end{proof}

\begin{corollary}
  \label{cor:8}                 
  Let \((M^{2m+1},\theta ,J,L_{\theta })\) be a Sasakian manifold of
  constant pseudo-Hermitian sectional curvature and
  let \(p\) and \(q\) be any two points of \(M\).
  Let \(\{\xi_{p},e_{1}, \cdots, e_{m}\), \(Je_{1}, \cdots, Je_{m}\}\in T_{p}(M)\) and
  \(\{\xi_{q},\widehat{e}_{1}, \cdots, \widehat{e}_{m},
  J\widehat{e}_{1}, \cdots, J \widehat{e}_{m}\}\in T_{q}(M)\)
  be arbitrary orthonormal bases.
  Then there exist neighborhoods \(U\) of \(p\) and \(\widehat{U}\) of \(q\),
  and a pseudo-Hermitian isometry \(f:U\rightarrow \widehat{U}\) such that \(df_{p}(\xi_{p})=\xi_{q}\),
  \(df_{p}(e_{i})=\widehat{e}_{i}\) and \(df_{p}(Je_{i})=J \widehat{e}_{i}\) (\(1\leq i\leq m\)).
\end{corollary}

A complete simply connected Sasakian manifold with
constant pseudo-Hermitian sectional curvature \(c\) is called a Sasakian space form
and denoted by \( M_{sk}^{2m+1}(c)\).
Under the mixed homothetic transformation of
\(g_{\theta }\) : \(g_{\theta }\mapsto \) \(g_{\mu L_{\theta },\mu \theta }\)
for a positive constant \(\mu \)
(which is called a \(\mathcal{D}\)-homothetic deformation in \cite{tanno1969sasak}),
the pseudo-Hermitian sectional curvature of \((M,g_{\mu L_{\theta},\mu \theta })\)
becomes \(c/\mu^{2}\).
Hence we only need to consider Sasakian space forms: \(M_{sk}^{2m+1}(0)\),
\(M_{sk}^{2m+1}(1)\) and \( M_{sk}^{2m+1}(-1)\).

\begin{example}[Sasakian space forms, cf. \cite{dragomir2007diff}, \cite{boyer2008sasaki}, \cite{blair2002riem}]
  \leavevmode
  \label{ex:3}                  
  \begin{enumerate}[(i)]
  \item The Heisenberg group \(H_{m}= \mathbb{C}^{m}\times \mathbb{R} \ \)is a Lie group with group law:
    \begin{equation*}
      (z,t)\cdot (w,s)
      =\biggl( z+w,t+s+2\operatorname{Im}(\sum_{j=1}^{m}z^{j}\overline{w}^{j})\biggr) ,
    \end{equation*}
    for any \((z,t)\),
    \((w,s)\in \) \( \mathbb{C}^{m}\times \mathbb{R} \).
    Set \(H_{(z,t)}^{1,0}(H_{m})
    =span\{\frac{\partial}{\partial z^j}+\sqrt{-1}\overline{z}^j\frac{\partial}{\partial t}\}_{j=1}^m\),
    where \(\frac{\partial }{\partial z^{j}}
    =\frac{1}{2}(\frac{\partial }{\partial x^{j}}-i \frac{\partial }{\partial y^{j}})\),
    \(z=(z^{1}, \cdots, z^{m})\in \mathbb{C}^{m}\).
    Let
    \begin{equation*}
      \theta =dt+\sqrt{-1}\sum_{j=1}^{m}\bigl( z^{j}d\overline{z}^{j}-\overline{z}^{j}dz^{j}\bigr) .
    \end{equation*}
    Then \((H_{m},\theta ,J,L_{\theta })\) is a Sasakian space form with
    pseudo-Hermitian sectional curvature \(0\).

  \item Let \(i:S^{2m+1}\hookrightarrow \mathbb{C}^{m+1}\)
    be the natural embedding of \((2m+1)\)-sphere.
    Then we have the CR sphere \((S^{2m+1},\theta ,J,L_{\theta })\)
    with the induced pseudo-Hermitian structure \((\theta ,J,L_{\theta })\) from \(\mathbb{C}^{m+1}\). 
    Then \((S^{2m+1},\theta ,J,L_{\theta })\) is a Sasakian space form
    with pseudo-Hermitian sectional curvature \(1\).

  \item Let \(B_{m}\) be the complex unit ball in \( \mathbb{C}^{m}\) with
    the Bergman metric \(ds_{B_{m}}^{2}\) of constant holomorphic sectional curvature \(-1\).
    Let \(\omega \) be the K\"{a}hler form of \( ds_{B_{m}}^{2}\) and
    \(\alpha \) be a \(1\)-form on \(B_{m}\) such that \(\omega =d\alpha\).
    Set \(D^{2m+1}=B\times \mathbb{R} \) and \(\theta =dt+\pi^{\ast }\alpha \),
    where \(\pi :D^{2m+1}\rightarrow B_{m} \) is the natural projection and \(t\in \mathbb{R} \).
    Set \(H(D)=\ker \theta \).
    Define an almost complex structure \(J_{b}\) on \(H(D)\) to be
    the horizontal lift of the almost complex structure \(J_{B_{m}}\) on \(B_{m}\). 
    Then \((D^{2m+1},\theta ,J,L_{\theta })\) is a Sasakian space form
    with pseudo-Hermitian sectional curvature \(-1\).
  \end{enumerate}
\end{example}

In \cite{tanno1969sasak},
Tanno showed that Sasakian space forms are isomorphic to the three examples listed above.
Notice that an isomorphism in \cite{tanno1969sasak} has the same meaning as a pseudo-Hermitian isometry in our terminology.
However,
Tanno's proof for his result used the real analyticity of \(M\) and \( g_{\theta }\).
We shall now give a new proof for this classification result.
The following lemma that is known in Riemannian geometry shows that local isometries have strong rigidity.

\begin{lemma}[cf. \cite{carmo1992riem}]
  \label{lem:11}                
  Let \(f_{i}:M\rightarrow N\), 
  \(i=1,2\), 
  be two local isometries of the (connected) Riemannian manifold \(M\) to the Riemannian manifold \(N\). 
  Suppose that there exists a point \(p\in M\) such that \(f_{1}(p)=f_{2}(p)\) and \((df_{1})_{p}=(df_{2})_{p}\). 
  Then \(f_{1}=f_{2}\).
\end{lemma}

\begin{lemma}[cf. \cite{dong2018comp}]
  \label{lem:10}                
  Let \((M^{2m+1},\theta ,J,L_{\theta })\) be a complete Sasakian manifold with \(K^{H}\leq 0\).
  Then \(\exp_{p}^{\nabla }:T_{p}(M)\rightarrow M\) is a covering map for any \( p\in M\).
  In particular,
  if \(M\) is simply connected,
  then \(\exp_{p}^{\nabla}:T_{p}(M)\rightarrow M\) is diffeomorphism.
\end{lemma}

\begin{remark}
  \label{rmk:5}                 
  Let \(\pi :\widetilde{M}\rightarrow (M,\theta ,J,h)\)
  be a topological covering map into a pseudo-Hermitian manifold.
  There is a unique smooth structure on \(\widetilde{M}\)
  such that \(\pi :\widetilde{M}\rightarrow M\) is a smooth covering map
  (cf. \cite{lee1988einst}).
  Then \((d \pi)_{\widetilde{p}}:T_{\widetilde{p}}(\widetilde{M})\rightarrow T_{p}(M)\)
  is a linear isomorphism for any \(\widetilde{p}\in \widetilde{M}\) and \(p\in M\).
  Set \(\widetilde{J}_{\widetilde{p}}
  = (d \pi)_{\widetilde{p}}^{-1}\circ J_p\circ (d \pi)_{\widetilde{p}}\),
  \(\widetilde{\theta }_{\widetilde{p}}=\pi_{\widetilde{p}}^{\ast }(\theta_{p})\)
  and
  \(\widetilde{h}_{\widetilde{p}}=\pi_{\widetilde{p}}^{\ast }(h_{p})\).
  Hence \((\widetilde{M},\widetilde{\theta }, \widetilde{J},\widetilde{h})\)
  becomes a pseudo-Hermitian manifold with the induced pseudo-Hermitian structure,
  so that \(\pi :\widetilde{M} \rightarrow M\) is a local pseudo-Hermitian isometry.
\end{remark}

\begin{theorem}
  \label{thm:12}                
  Let \((M^{2m+1},\theta ,J,L_{\theta })\) be a complete Sasakian manifold with
  constant pseudo-Hermitian curvature \(K_{\theta }\).
  Then there is a pseudo-Hermitian isometry from the universal covering \(\widetilde{M}\) of \(M\) to
  \begin{enumerate}[(a)]
  \item \(H_{m}\), if \(K_{\theta }=0\);

  \item \(D^{2m+1}\), if \(K_{\theta }=-1\);

  \item \(S^{2m+1}\), if \(K_{\theta }=1\).
  \end{enumerate}
\end{theorem}

\begin{proof}
  We first consider the cases (a) and (b),
  and denote \(H_{m}\) as well as \( D^{2m+1}\) by \(\Sigma \).
  Choose any points \(p\in \Sigma \),
  \(\widetilde{p}\in \widetilde{M}\)
  and a pseudo-Hermitian linear isometry
  \(\rho :T_{p}(\Sigma)\rightarrow T_{\widetilde{p}}(\widetilde{M})\).
  By \eqref{4.14},
  we see that both \(\Sigma \) and \(\widetilde{M}\) have non-positive horizontal sectional curvature.
  According to Lemma \ref{lem:10},
  both \(\exp_{p}^{\nabla }:T_{p}(\Sigma)\rightarrow \Sigma \) and
  \(\exp_{\widetilde{p}}^{\widetilde{\nabla }}:T_{\widetilde{p}}(\widetilde{M})\rightarrow \widetilde{M}\)
  are diffeomorphism.
  It follows from Corollary \ref{cor:7} that
  \(f=\exp_{\widetilde{p}}^{\widetilde{\nabla }}\circ \rho \circ
  (\exp_{p}^{\nabla })^{-1}:\Sigma \rightarrow \widetilde{M}\)
  is a global pseudo-Hermitian isometry from \(\Sigma \) to \( \widetilde{M}\),
  and so \(f^{-1}:\widetilde{M}\rightarrow \Sigma \) is also a pseudo-Hermitian isometry.
  This proves (a) and (b).

  For the case (c), 
  we fix points \(p\in S^{2m+1}\), 
  \(\widetilde{p}\in \widetilde{M}\) and a pseudo-Hermitian linear isometry \(\rho :T_{p}(S^{2m+1})\rightarrow T_{\widetilde{p}}(\widetilde{M})\). 
  Let \(B(x_{0};r)\) denote the metric ball with respect to the canonical metric \(g_{\theta }\) on \(S^{2m+1}\), 
  with center \(x_{0}\in S^{2m+1}\) and radius \(r\). 
  Since \(S^{2m+1}\) is compact, 
  there is a positive number \(R_{0}\) such that for any \(x\in S^{2m+1}\), 
  \(B(x;R_{0})\) is contained in a normal neighborhood of \(x\). 
  Notice that the normal neighborhood is defined with respect to the Tanaka-Webster connection \(\nabla\). 
  Choose a positive integer \(n\) such that \(a:=\pi /n<R_{0}\). 
  Now we define 
  \begin{equation*}
    f_{1}=\exp_{\widetilde{p}}^{\widetilde{\nabla }}\circ \rho \circ (\exp_{p}^{\nabla })^{-1}:\overline{B(p;a)}\rightarrow \widetilde{M}.
  \end{equation*}
  By Corollary \ref{cor:7}, 
  \(f_{1}\) is a local pseudo-Hermitian isometry. 
  For any point \(x\in \partial B(p;a)\), 
  set \(\widetilde{x}=f_{1}(x)\) and \(\rho_{x}=(df_{1})_{x}\). 
  Similarly, 
  we define the following map 
  \begin{equation*}
    h_{x}=\exp_{\widetilde{x}}^{\widetilde{\nabla }}\circ \rho_{x}\circ (\exp_{x}^{\nabla })^{-1}:\overline{B(x;a)}\rightarrow \widetilde{M}
  \end{equation*}
  which is also a local pseudo-Hermitian isometry. 
  This gives a family of local pseudo-Hermitian isometries \(\{h_{x}:x\in \partial B(p;a)\}\). 
  Clearly \(h_{x}(x)=\widetilde{x}=f_{1}(x)\) and \((dh_{x})_{x}=\rho_{x}=(df_{1})_{x}\). 
  It follows from Lemma \ref{lem:11} that \(f_{1}=h_{x}\) on \(B(p;a)\cap B(x;a)\). 
  If \(B(x;a)\cap B(y;a)\neq \varnothing\) for any two points \(x,y\in \partial B(p;a)\), 
  it is easy to see that \(h_{x}(q)=h_{y}(q)=f_{1}(q)\) and \((dh_{x})_{q}=(dh_{y})_{q}=(df_{1})_{q}\) for any \(q\in B(p;a)\cap B(x;a)\cap B(y;a)\). 
  Then Lemma \ref{lem:11} implies that \(h_{x}=h_{y}\) on \(B(x;a)\cap B(y;a)\). 
  Therefore we can define a map $f_{2}:\overline{B(p;2a)}\rightarrow \widetilde{M}$ by
  \begin{equation*}
    f_{2}(q)=
    \begin{cases}
      f_{1}(q), & \text{if } \ q\in \overline{B(p;a)} \\ 
      h_{x}(q), & \text{if } \ q\in \overline{B(x;a)} \text{ for some }x\in \partial B(p;a)
    \end{cases}
  \end{equation*}
  which is a local pseudo-Hermitian isometry. 
  We can carry this process on inductively, 
  obtaining \(f_{i+1}\) from \(f_{i}\), 
  \(i=1,...,n-1\). 
  Finally, 
  we get a global map \(f_{n}:S^{2m+1}\rightarrow \widetilde{M}\). 
  Clearly \(f_{n}\) is a local pseudo-Hermitian isometry, 
  and thus a local diffeomorphism. 
  By the compactness of \(S^{2m+1}\), 
  \(f_{n}\) is a covering map. 
  Since \(\widetilde{M}\) is simply connected, 
  we conclude that \(f_{n}:S^{2m+1}\rightarrow \widetilde{M}\) is a diffeomorphism, 
  and therefore a global pseudo-Hermitian isometry.
\end{proof}

\begin{remark}
  The arguments for Cartan-type theorems on pseudo-Hermitian manifolds in this section
  are modifications of those for Cartan's theorem on Riemannian manifolds in \cite{carmo1992riem}.
\end{remark}

It would be interesting to investigae geometric analysis problems on a pseudo-Hermitian manifold \((M,\theta ,J,h)\). 
We will consider some of these problems in future work. 
Before concluding this paper, 
we would like to discuss the following question related to the final section: A complete simply connected pseudo-K\"{a}hler manifold with \(\tau =0\) and constant pseudo-Hermtian sectional curvature is called a \emph{pseudo-K\"{a}hler space form}. 
Obviously, 
a mixed homothetic transformation (defined in \S \ref{sec:pseudo-herm-struct}) tranforms a pseudo-K\"{a}hler space form into another pseudo-K\"{a}hler space form. 
Consequently, 
there are examples of pseudo-K\"{a}hler space forms obtained from Sasakian space forms through mixed homothetic transformations. 
Are there any other examples of pseudo-K\"{a}hler space forms ? If there are such examples, 
can we give a classification of them?

\printbibliography

Yuxin Dong

School of Mathematical Sciences

Fudan University

Shanghai, 200433, P. R. China

yxdong@fudan.edu.cn

\vspace{20pt}

Yibin Ren

School of Mathematical Sciences

Zhejiang Normal University

Jinhua 321004, P.R. China

allenryb@outlook.com

\end{document}